\documentclass[pagesize,pdftex]{scrartcl}
\usepackage{latexsym} 
\usepackage[intlimits]{amsmath}
\usepackage{amsthm}
\usepackage{amsfonts}
\usepackage{amssymb}
\usepackage{amsxtra}
\usepackage{amscd}
\usepackage{ifthen}
\usepackage{graphicx}
\usepackage[shortlabels]{enumitem}
\usepackage{mathrsfs}
\usepackage[pagebackref=true]{hyperref}
\usepackage{mathtools}
\usepackage{MnSymbol}

\hypersetup{
pdfauthor={Thomas Eiter, Mads Kyed},
pdftitle={Viscous flow around a rigid body performing a time-periodic motion},
breaklinks=true,
colorlinks=true,
linkcolor=blue,
citecolor=blue,
urlcolor=blue,
filecolor=blue,
}

\numberwithin{equation}{section} 
\setkomafont{title}{\normalfont}

\newenvironment{pdeq}{ \left\{ \begin{aligned}}{\end{aligned}\right.}

\newcommand{\eqrefsub}[2]{\ensuremath{\eqref{#1}_{#2}}}
\newcommand{\np}[1]{(#1)}
\newcommand{\nb}[1]{[#1]}
\newcommand{\bp}[1]{\big(#1\big)}
\newcommand{\bb}[1]{\big[#1\big]}
\newcommand{\Bp}[1]{\bigg(#1\bigg)}

\newcommand{\Bb}[1]{\bigg[#1\bigg]}
\newcommand{\calb}{{\mathcal B}}

\newcommand{\caln}{{\mathcal N}}

\newcommand{\calp}{{\mathcal P}}

\newcommand{\cals}{{\mathcal S}}
\newcommand{\calt}{{\mathcal T}}

\newcommand{\calx}{{\mathcal X}}

\newcommand{\R}{\mathbb{R}}
\newcommand{\C}{\mathbb{C}}
\newcommand{\Z}{\mathbb{Z}}
\newcommand{\N}{\mathbb{N}}
\DeclareMathOperator{\e}{e}

\DeclareMathOperator{\id}{Id}
\DeclareMathOperator{\Div}{div}

\DeclareMathOperator{\supp}{supp}

\DeclareMathOperator{\rot}{rot}
\newcommand{\bogopr}{\mathfrak{B}}

\DeclareMathOperator{\realpart}{Re}

\newcommand{\embeds}{\hookrightarrow}
\DeclareMathOperator{\vecspan}{span}

\newcommand{\ra}{\rightarrow}

\newcommand{\set}[1]{\ensuremath{\{#1\}}}

\newcommand{\setc}[2]{\ensuremath{\{#1\ \mid\ #2\}}}
\newcommand{\setcl}[2]{\ensuremath{\bigl\{#1\ \big\mid\ #2\bigr\}}}
\newcommand{\setcL}[2]{\ensuremath{\biggl\{#1\ \bigg\mid\ #2\biggr\}}}

\newcommand{\closure}[2]{\overline{#1}^{#2}}

\newcommand{\ball}{\mathrm{B}}

\renewcommand{\restriction}[2]{#1\big | _{#2}}
\newcommand{\proj}{\calp}
\newcommand{\projcompl}{\calp_\bot}

\newcommand{\projh}{\proj_\mathrm{H}}
\newcommand{\grp}{G}
\newcommand{\dualgrp}{\widehat{G}}
\newcommand{\dual}[1]{\widehat{#1}}

\newcommand{\torus}{{\mathbb T}}

\newcommand{\transpose}{\top}

\newcommand{\idmatrix}{\mathrm{I}}
\newcommand{\Rn}{{\R^n}}

\newcommand{\ddt}{\frac{{\mathrm d}}{{\mathrm d}t}}
\newcommand{\tddt}{\tfrac{{\mathrm d}}{{\mathrm d}t}}
\newcommand{\grad}{\nabla}

\newcommand{\dx}{{\mathrm d}x}

\newcommand{\dt}{{\mathrm d}t}

\newcommand{\dS}{{\mathrm d}S}
\newcommand{\dxi}{{\mathrm d}\xi}

\newcommand{\nvec}{\mathrm{n}}

\newcommand{\SR}{\mathscr{S}}

\newcommand{\TDR}{\mathscr{S^\prime}}

\newcommand{\FT}{\mathscr{F}}
\newcommand{\iFT}{\mathscr{F}^{-1}}

\newcommand{\AR}{\mathrm{A}}
\newcommand{\norm}[1]{\lVert#1\rVert}

\newcommand{\norml}[1]{\bigl\lVert#1\bigr\rVert}
\newcommand{\normL}[1]{\Bigl\lVert#1\Bigr\rVert}
\newcommand{\snorm}[1]{{\lvert #1 \rvert}}
\newcommand{\snorml}[1]{{\bigl\lvert #1 \big\rvert}}

\newcommand{\WSR}[2]{\mathrm{W}^{#1,#2}} 
\newcommand{\WSRN}[2]{\mathrm{W}^{#1,#2}_0}

\newcommand{\WSRloc}[2]{\mathrm{W}^{#1,#2}_{\mathrm{loc}}} 
\newcommand{\CR}[1]{\mathrm{C}^{#1}}  
\newcommand{\LR}[1]{\mathrm{L}^{#1}}
 
\newcommand{\lR}[1]{\ell^{#1}}

\newcommand{\LRloc}[1]{\mathrm{L}^{#1}_{\mathrm{loc}}} 
\newcommand{\CRi}{\CR \infty}
\newcommand{\CRci}{\CR \infty_0}

\newcommand{\LRsigma}[1]{\mathrm{L}^{#1}_{\sigma}}

\newcommand{\CRcisigma}{\CR{\infty}_{0,\sigma}}

\newcommand{\nsnonlinb}[2]{#1\cdot\grad #2}

\newcommand{\vvel}{v}
\newcommand{\vpres}{p}

\newcommand{\Vvel}{V}

\newcommand{\wvel}{w}
\newcommand{\wpres}{\mathfrak{q}}

\newcommand{\uvel}{u}

\newcommand{\uvels}{\uvel_{\mathrm{s}}}
\newcommand{\uvelp}{\uvel_{\mathrm{p}}}

\newcommand{\upress}{\upres_{\mathrm{s}}}
\newcommand{\upresp}{\upres_{\mathrm{p}}}

\newcommand{\upres}{\mathfrak{p}}

\newcommand{\Uvel}{U}

\newcommand{\Upres}{\mathfrak{P}}

\newcommand{\rotterm}[1]{\np{ \eone\wedge #1 - \nsnonlinb{\eone\wedge x}{#1}}}

\newcommand{\rottermsimple}[1]{ \eone\wedge #1 - \nsnonlinb{\eone\wedge x}{#1} }

\newcommand{\rotderterm}[1]{\np{\partial_t #1 + \eone\wedge #1 - \nsnonlinb{\eone\wedge x}{#1} }}
\newcommand{\rotdertermsimple}[1]{\partial_t #1 + \eone\wedge #1 - \nsnonlinb{\eone\wedge x}{#1} }
\newcommand{\rotdertermFT}[1]{\np{ik #1 + \eone\wedge #1 - \nsnonlinb{\eone\wedge x}{#1} }}
\newcommand{\rotdertermFTseq}[1]{\np{ik_j #1 + \eone\wedge #1 - \nsnonlinb{\eone\wedge x}{#1} }}
\newcommand{\rotdertermFTsimple}[1]{ik #1 + \eone\wedge #1 - \nsnonlinb{\eone\wedge x}{#1} }
\newcommand{\rotdertermFTsimpleseq}[1]{ik_j #1 + \eone\wedge #1 - \nsnonlinb{\eone\wedge x}{#1} }
\newcommand{\uf}{f}

\newcommand{\vf}{F}

\newcommand{\wf}{h}
\newcommand{\tin}{\text{in }}

\newcommand{\ton}{\text{on }}

\newcommand{\tfor}{\text{for }}

\newcommand{\half}{\frac{1}{2}}

\renewcommand{\epsilon}{\varepsilon}
\renewcommand{\phi}{\varphi}
\newcommand{\rey}{\lambda}

\newcommand{\veltrans}{\alpha}
\newcommand{\tay}{\omega}

\newcommand{\per}{\calt}
\newcommand{\iper}{\frac{1}{\per}}

\newcommand{\eone}{\e_1}

\newcommand{\rotmatrix}{Q}

\newcommand{\spanspacemn}{X^m_n}

\newcommand{\cutoff}{\chi}

\newcommand{\change}[1]{}

\newcommand{\barycenter}{x_\mathrm{C}}
\newcommand{\newCCtr}[2][d]{
\newcounter{#2}\setcounter{#2}{0}
\expandafter\xdef\csname kyedtheconst#2\endcsname{#1}
}
\newcommand{\Cc}[2][nolabel]{
\stepcounter{#2}
\expandafter\ensuremath{\csname kyedtheconst#2\endcsname_{\arabic{#2}}}
\ifthenelse{\equal{#1}{nolabel}}
{}
{\expandafter\xdef\csname kyedconst#1\endcsname
{\expandafter\ensuremath{\csname kyedtheconst#2\endcsname_{\arabic{#2}}}}}
}
\newcommand{\Ccn}[2][nolabel]{
\expandafter\ensuremath{\csname kyedtheconst#2\endcsname}
\ifthenelse{\equal{#1}{nolabel}}
{}
{\expandafter\xdef\csname kyedconst#1\endcsname
{\expandafter\ensuremath{\csname kyedtheconst#2\endcsname}}}
}
\newcommand{\CcSetCtr}[2]{
\setcounter{#1}{#2}
}
\newcommand{\Cclast}[1]{
\expandafter\ensuremath{\csname kyedtheconst#1\endcsname_{\arabic{#1}}}
}
\newcommand{\Ccllast}[1]{
\addtocounter{#1}{-1}
\expandafter\ensuremath{\csname kyedtheconst#1\endcsname_{\arabic{#1}}}
\addtocounter{#1}{1}
}
\newcommand{\const}[1]{
\expandafter{\ifcsname kyedconst#1\endcsname
  \csname kyedconst#1\endcsname
\else
  \errmessage{Undefined Kyedconstant #1.}%
\fi}
}

\renewcommand{\eqrefsub}[2]{\eqref{#1}\textsubscript{#2}}

\theoremstyle{plain}
\newtheorem{thm}{Theorem}[section]

\newtheorem{lem}[thm]{Lemma}
\newtheorem{prop}[thm]{Proposition}

\theoremstyle{remark}
\newtheorem{rem}[thm]{Remark}

\begin{document}
\title{Viscous flow around a rigid body performing a time-periodic motion}

\author{
Thomas Eiter\\ 
Fachbereich Mathematik\\
Technische Universit\"at Darmstadt\\
Schlossgartenstr. 7, 64289 Darmstadt, Germany\\
Email: {\texttt{eiter@mathematik.tu-darmstadt.de}}
\and
Mads Kyed\\ 
Flensburg University of Applied Sciences\\
Kanzleistra\ss e 91-93, 24943 Flensburg, Germany \\
Email: {\texttt{mads.kyed@hs-flensburg.de}}
}

\date{\today}
\maketitle

\begin{abstract}
The equations governing the flow of a viscous incompressible fluid 
around a rigid body that performs a prescribed time-periodic motion
with constant axes of translation and rotation are investigated.
Under the assumption that the period and the angular velocity of the prescribed rigid-body motion are compatible,
and that the mean translational velocity is non-zero, 
existence of a time-periodic solution is established.
The proof is based on an appropriate linearization, which is examined 
within a setting of absolutely convergent Fourier series.
Since the corresponding resolvent problem is ill-posed in classical Sobolev spaces, 
a linear theory is developed in a framework of homogeneous Sobolev spaces.
\end{abstract}

\noindent\textbf{MSC2010:} Primary 35Q30, 35B10, 76D05, 76D07, 76U05.\\
\noindent\textbf{Keywords:} Navier-Stokes, Oseen flow, time-periodic solutions, rotating obstacles.

\newCCtr[C]{C}
\newCCtr[M]{M}
\newCCtr[\epsilon]{eps}
\CcSetCtr{eps}{-1}
\newCCtr[c]{c}
\let\oldproof\proof
\def\proof{\CcSetCtr{c}{-1}\oldproof}

\section{Introduction}
We investigate the fluid flow past a rigid body $\calb$ that moves through an infinite three-dimensional liquid reservoir with prescribed velocity
\[
\Vvel(t,x)=\xi(t)+\eta\wedge \np{x-\barycenter(t)}
\]
with respect to its center of mass $\barycenter$.
Here $t\in\R$ and $x\in\R^3$ denote time and spatial variable, respectively,
$\xi\coloneqq\tddt\barycenter$ the translation velocity and
$\eta$ the angular velocity of $\calb$ with respect to its center of mass.
We consider only the case where the angular velocity $\eta$ is constant, but the translation velocity $\xi$ may depend on time.
In a frame attached to the body, with origin at its center of mass $\barycenter$, the motion of an incompressible Navier--Stokes fluid around $\calb$
that adheres to $\calb$ at the boundary
is described by the equations 
\begin{align}\label{sys:RotatingOseen_FixedDomain}
\begin{pdeq}
\rho\bp{\partial_t \uvel + \eta\wedge\uvel-\eta\wedge x \cdot\grad\uvel 
-\xi\cdot\grad\uvel+ \uvel\cdot\grad\uvel} 
&= f +\mu \Delta \uvel - \grad \upres 
&& \tin \R\times\Omega, \\
\Div\uvel&=0
&& \tin \R\times\Omega, \\
\uvel&=\xi+\eta\wedge x
&& \ton \R\times\partial\Omega, \\
\lim_{\snorm{x}\to\infty} \uvel(t,x) &= 0
&& \tfor t\in \R;
\end{pdeq}
\end{align}
see \cite[Section 1]{Galdi_OnTheMotionRigidBodyInViscousLiquid_2002}.
Here $\Omega\coloneqq\R^3\setminus\overline{\calb}$ is the exterior domain surrounding $\calb$, 
and $\R$ represents the time axis.
The functions $\uvel\colon\R\times\Omega\to\R^3$ and $\upres\colon\R\times\Omega\to\R$ 
describe velocity and pressure fields of the fluid.
The constants $\rho>0$ and $\mu>0$ denote density and viscosity, respectively.
For the sake of generality, we additionally consider 
an external body force $f\colon\R\times\Omega\to\R^3$.

In this paper, we investigate a configuration where the rigid body $\calb$ translates periodically
with some prescribed time period $\per>0$. 
More precisely, we assume the data
\[
\xi(t+\per)=\xi(t),\qquad
\uf(t+\per,x)=\uf(t,x)
\]
to be $\per$-time-periodic
As the main theorem we show existence of a solution 
$\np{\uvel,\upres}$ to \eqref{sys:RotatingOseen_FixedDomain}
that shares this time periodicity.

We consider a prescribed motion of $\calb$ where the axes of translation and rotation
do not vary over time and are parallel.
Without loss of generality, both are directed along the $x_1$-axis such that
\[
\xi(t)=\veltrans(t)\eone, \qquad \eta=\tay\eone
\]
for some $\per$-periodic function $\alpha\colon\R\to\R$ and a constant $\tay\in\R$.
Note that, at least in the case where $\xi$ is time-independent, this 
assumption can be made without loss of generality
as long as $\xi\cdot\eta\neq 0$ due to the Mozzi--Chasles theorem.

We assume that the mean translational velocity 
of the body over one time period is non-zero:
\begin{align}\label{eq:RotatingOseen_MeanTranslationNot0}
\rey\coloneqq\iper\int_0^\per \veltrans(t)\,\dt\neq 0.
\end{align} 
The case of vanishing mean translational velocity shall not be treated here. 
Not only does the fluid flow exhibit different physical properties when \eqref{eq:RotatingOseen_MeanTranslationNot0} is not satisfied,   
due to the absence of a wake region in this case, also the mathematical properties of the linearization 
of \eqref{sys:RotatingOseen_FixedDomain} differ significantly.
If \eqref{eq:RotatingOseen_MeanTranslationNot0} is satisfied,
the linearization of \eqref{sys:RotatingOseen_FixedDomain} 
is a time-periodic generalized Oseen system,
for which we shall establish suitable $\LR{q}$ estimates 
in order to show existence of a solution to \eqref{sys:RotatingOseen_FixedDomain}.
If \eqref{eq:RotatingOseen_MeanTranslationNot0} is not satisfied,
the linearization of \eqref{sys:RotatingOseen_FixedDomain} 
is a time-periodic generalized Stokes system,
for which similar estimates cannot be derived.
In this case, 
problem \eqref{sys:RotatingOseen_FixedDomain} thus has to be approached in a different way,
which has recently been done by \textsc{Galdi} \cite{Galdi_ViscousFlowPastBodyTPZeroAverage_2019}.

Since the case $\eta=0$ was treated in \cite{GaldiKyed_TPflowViscLiquidpBody_2018}, we consider only the case $\eta\neq 0$ in the following. 
Observe that $\eta\wedge x\cdot\grad$ is then a differential operator with \emph{unbounded} coefficient.
Therefore, the linearization of \eqref{sys:RotatingOseen_FixedDomain} 
cannot be treated as a lower-order perturbation
of the time-periodic Oseen problem, even if $\eta$ is ``small''. 
In particular, as we will see below, also the corresponding resolvent problem 
requires an analysis in a different functional setting.
This behavior reflects the properties of the corresponding stationary problem
(see \cite[Chapter VIII]{GaldiBookNew}),
which can be regarded as a special case of the time-periodic problem.
In order to find a framework in which the time-periodic generalized Oseen problem is well posed,
we employ the idea from
\cite{KyedGaldi_asplqesoserfII,KyedGaldi_asplqesoserfI}, where
the steady-state problem corresponding to \eqref{sys:RotatingOseen_FixedDomain} was considered,
and the rotation term
$\eta\wedge\uvel-\eta\wedge x \cdot\grad\uvel$
was handled by a change of coordinates into a non-rotating frame.
This procedure, however, merely yields suitable estimates for time-periodic solutions
when the change of coordinates maintains the time periodicity of the involved functions.
This is the case if the angular velocity $\tay$
is an integer multiple of the angular frequency $2\pi/\per$ of the time-periodic data.
For simplicity, we assume
\begin{equation}\label{eq:FrequenciesCoincide}
\tay=2\pi/\per.
\end{equation}
This condition means that during one period 
the rigid body completes one full revolution.
In other words, the rotation and the time-periodic data,
which may be regarded as two different sources of time-periodic forcing, 
have to be compatible.

The equations governing the fluid flow around 
a rigid body that performs a prescribed rigid motion 
has been studied by many researchers during the last decades.
The first attempts of a rigorous mathematical treatment can be dated back
to the fundamental works of \textsc{Oseen} \cite{oseenbook}, \textsc{Leray} \cite{Leray1933, Leray1934b}
and \textsc{Lady\v{z}henskaya} \cite{Ladyzenskaya_InvestigationNSStatMotionIncFluid_1959, LadyzhenskayaBook}. 
In a short note,
\textsc{Serrin} \cite{Serrin_PeriodicSolutionsNS1959} proposed the examination of the corresponding time-periodic configuration,
and  \textsc{Prodi} \cite{Prodi1960}, \textsc{Yudovich} \cite{Yudovich60} 
and \textsc{Prouse} \cite{Prouse63}
initiated the study of time-periodic Navier--Stokes flow in bounded domains.
Through the years, this investigation has been continued and extended 
to other types of domains and fluid-flow problems by several authors, see for example
\cite{KanielShinbrot67,Takeshita69,Morimoto1972,MiyakawaTeramoto82,Teramoto1983,
Maremonti_TimePer91,Maremonti_HalfSpace91,MaremontiPadula96,
KozonoNakao96, Yamazaki2000, GaldiSohr2004, 
GaldiSilvestre_ExistenceTPSolutionsNSAroundMovingBody_2006,
GaldiSilvestre_MotionRigidBodyNSLiquidTPForce_2009,
Taniuchi2009, BaalenWittwer2011, Silvestre_TPFiniteKineticEnergy12, GaldiTP2D13, Kyed_habil,
Kyed_ExistenceRegularityTPNS2014, Kyed_FundsolTPStokes2016, Nguyen2014, GeissertHieberNguyen_TP2016,
GaldiKyed_TPSolNS3D:AsymptoticProfile,  
EiterKyed_etpfslns_2018, GaldiKyed_TPflowViscLiquidpBody_2018}.
We also refer to \cite{GaldiKyed_TPSolNSE_Handbook2018} for a more detailed overview.
The time-periodic problem \eqref{sys:RotatingOseen_FixedDomain}
was object of research both in the article by
\textsc{Galdi} and \textsc{Silvestre}
\cite{GaldiSilvestre_ExistenceTPSolutionsNSAroundMovingBody_2006},
who established existence of time-periodic solutions 
in an $\LR{2}$ framework by a Galerkin approach, and in the article by
\textsc{Geissert}, \textsc{Hieber} and \textsc{Nguyen} \cite{GeissertHieberNguyen_TP2016},
who proved existence of mild time-periodic solutions within a setting of \emph{weak} $\LR{q}$ spaces
by means of semigroup theory for $\xi$ constant.
As the main novelty of the present paper, we present a proof of existence of strong solutions 
to \eqref{sys:RotatingOseen_FixedDomain} in an $\LR{q}$ setting.

Our approach is based on the analysis of 
the linearization of \eqref{sys:RotatingOseen_FixedDomain}
and the associated resolvent problem
\begin{align}\label{sys:RotatingOseen_resolvent_general}
\begin{pdeq}
is\vvel+\tay\rotterm{\vvel}- \Delta \vvel - \rey \partial_1 \vvel + \grad \vpres &= \vf
&& \tin \Omega, \\
\Div\vvel&=0
&& \tin \Omega, \\
\vvel&=0
&& \ton \partial\Omega
\end{pdeq}
\end{align}
for suitable $s\in\R$ and $\vf\in\LR{q}(\Omega)^3$, $1<q<\infty$.
At first glance, it seems reasonable to regard \eqref{sys:RotatingOseen_resolvent_general}
as a resolvent problem $(is-A)\vvel=\vf$ for a closed operator 
$A$ on the space of solenoidal vector-fields in $\LR{q}(\Omega)^3$.
However, the spectral analysis in this setting, which was carried out 
by \textsc{Farwig} and \textsc{Neustupa}
\cite{FarwigNeustupa_SpectrumOseenRotatingBody,FarwigNeustupa_SpectralPropertiesLq_2010}, 
reveals that $is$, $s\in\R$, belongs to the
spectrum of $A$ when $s\in \tay\Z$, which turn out to be exactly those values of $s$ that are required to be in the resolvent of the operator
in order to obtain a well-posed time-periodic problem.
Instead, we propose to investigate the problem in homogeneous Sobolev spaces.
Although it is merely possible to derive the \emph{non-classical}  
resolvent estimate \eqref{est:RotatingOseen_resolvent} in this setting
(see Theorem \ref{thm:RotatingOseen_resolvent} below),
we are nevertheless able to conclude a suitable solution theory for 
the linearization of \eqref{sys:RotatingOseen_FixedDomain}.
To this end, we shall employ a framework of functions with absolutely convergent Fourier series.
Finally, a fixed-point argument yields the existence of a solution to
the nonlinear problem \eqref{sys:RotatingOseen_FixedDomain} 
when the data $\uf$, $\xi$ and $\eta$ are ``sufficiently small''.

\section{Main results}\label{sec:Reformulation}

In virtue of \eqref{eq:RotatingOseen_MeanTranslationNot0} 
we may assume $\rey>0$ without loss of generality, and by \eqref{eq:FrequenciesCoincide}
we have $\tay=2\pi/\per>0$.
To reformulate \eqref{sys:RotatingOseen_FixedDomain} in a non-dimensional way,
we let the diameter $d>0$ of $\calb$ serve as a characteristic length scale.
We introduce the Reynolds number $\rey'\coloneqq\rey \rho d/\mu$ and the 
Taylor number $\tay'\coloneqq\tay \rho d^2/\mu$, and
the non-dimensional time and spatial variables 
$t'=\tay t$ and $x'=x/d$. 
In particular, $\Omega$ is transformed to
$
\Omega'
\coloneqq
\setc{x/d}{x \in \Omega}.
$
We define
$\veltrans'(t')\coloneqq\veltrans(t)\rho d/\mu$
and the non-dimensional functions
\[
\uvel'(t',x')\coloneqq\frac{\rho d}{\mu}\uvel(t,x), 
\qquad 
\upres'(t',x')\coloneqq\frac{\rho d^2}{\mu^2}\upres(t,x), 
\qquad
f'(t',x')\coloneqq\frac{\rho d^3}{\mu^2}f(t,x),
\]
which are time-periodic with period $\per'=2\pi$
and can thus be identified with functions on the torus group $\torus=\R/2\pi\Z$ with respect to time. 
Expressing \eqref{sys:RotatingOseen_FixedDomain} in these new quantities and
omitting the primes,
we obtain the non-dimensional formulation 
\begin{align}\label{sys:RotatingOseen_Dimensionless}
\begin{pdeq}
\tay\np{\partial_t\uvel+\rottermsimple{\uvel} }
- \veltrans \partial_1 \uvel + \uvel\cdot\grad\uvel &= f + \Delta \uvel - \grad \upres
&& \tin \torus\times\Omega, \\
\Div\uvel&=0
&& \tin \torus\times\Omega, \\
\uvel&=\veltrans\eone + \tay \eone \wedge x
&& \ton \torus\times\partial\Omega, \\
\lim_{\snorm{x}\to\infty} \uvel(t,x) &= 0
&& \tfor t\in \torus.
\end{pdeq}
\end{align}
Our analysis of \eqref{sys:RotatingOseen_Dimensionless} is based on the study of the linear 
time-periodic problem
\begin{align}\label{sys:RotatingOseen_linear}
\begin{pdeq}
\tay\rotderterm{\uvel}- \Delta \uvel - \rey \partial_1 \uvel + \grad \upres &= f
&& \tin \torus\times\Omega, \\
\Div\uvel&=0
&& \tin \torus\times\Omega, \\
\uvel&=0
&& \ton \torus\times\partial\Omega,
\end{pdeq}
\end{align}
and of the corresponding resolvent problem
\begin{align}\label{sys:RotatingOseen_resolvent}
\begin{pdeq}
\tay\rotdertermFT{\vvel}- \Delta \vvel - \rey \partial_1 \vvel + \grad \vpres &= \vf
&& \tin \Omega, \\
\Div\vvel&=0
&& \tin \Omega, \\
\vvel&=0
&& \ton \partial\Omega
\end{pdeq}
\end{align}
for $k\in\Z$.
For the latter we shall derive the following well-posedness result.

\begin{thm}\label{thm:RotatingOseen_resolvent}
Let $\Omega\subset\R^3$ be an exterior domain of class $\CR{3}$.
Let $q\in(1,2)$, $k\in\Z$ and $\rey,\,\tay, \,\theta,\,B>0$ with $\rey^2\leq\theta\tay\leq B$.
For every $\vf\in\LR{q}(\Omega)^3$
there exists a solution 
$\np{\vvel,\vpres}\in\WSRloc{2}{q}(\overline{\Omega})^3\times\WSRloc{1}{q}(\overline{\Omega})$ 
to \eqref{sys:RotatingOseen_resolvent}
subject to the estimate 
\begin{align}\label{est:RotatingOseen_resolvent}
\begin{split}
\tay&\norm{\rotdertermFTsimple{\vvel}}_{q}
+\norm{\grad^2\vvel}_{q}
+\rey\norm{\partial_1\vvel}_{q}
\\
&\qquad\qquad\qquad+\rey^{1/2}\norm{\vvel}_{s_1}
+\rey^{1/4}\norm{\grad\vvel}_{s_2} 
+\norm{\grad\vpres}_{q}
\leq\Cc[const:RotatingOseen_linear]{C}
\norm{\vf}_{q}
\end{split}
\end{align}
for a constant
$\const{const:RotatingOseen_linear}
=\const{const:RotatingOseen_linear}\np{\Omega,q,\rey,\tay}>0$
and $s_1=2q/\np{2-q}$, $s_2=4q/\np{4-q}$.
Additionally, if $\np{\wvel,\wpres}$ 
is another solution to \eqref{sys:RotatingOseen_resolvent}
in the function class defined by the norms on the left-hand side of \eqref{est:RotatingOseen_resolvent},
then $\vvel=\wvel$, and $\vpres-\wpres$ is a constant.
Moreover, if $q\in(1,\frac{3}{2})$, then the constant $\const{const:RotatingOseen_linear}$
can be chosen independently of $\rey$ and $\tay$ such that
$\const{const:RotatingOseen_linear}
=\const{const:RotatingOseen_linear}(\Omega,q,\theta,B)$.
\end{thm}

Note that for $k=0$
we recover the well-known $\LR{q}$ theory for the corresponding stationary problem; see \cite[Theorem VIII.8.1]{GaldiBookNew}.

In order to transfer estimate \eqref{est:RotatingOseen_resolvent}
to the time-periodic setting without
losing information on the dependencies of the constant $\const{const:RotatingOseen_linear}$,
we work within spaces $\AR(\torus;X)$ of absolutely convergent $X$-valued Fourier series
for suitable Banach spaces $X$; see \eqref{eq:DefOfA} below. 
We establish the following solution theory for 
the time-periodic problem \eqref{sys:RotatingOseen_linear}.

\begin{thm}\label{thm:RotatingOseen_linear}
Let $\Omega\subset\R^3$ be an exterior domain of class $\CR{3}$.
Let $q\in(1,2)$ and $\rey,\,\tay,\,\theta,\,B>0$ with $\rey^2\leq\theta\tay\leq B$.
For every 
$f\in\AR(\torus;\LR{q}(\Omega))^3$
there exists a solution 
$\np{\uvel,\upres}$
to \eqref{sys:RotatingOseen_linear}
subject to the estimate
\begin{align}\label{est:RotatingOseen_linear}
\begin{split}
\tay&\norm{\rotdertermsimple\uvel}_{\AR(\torus;\LR{q}(\Omega))}
+\norm{\grad^2\uvel}_{\AR(\torus;\LR{q}(\Omega))}
+\rey\norm{\partial_1\uvel}_{\AR(\torus;\LR{q}(\Omega))}
\\
& \quad\ 
+\rey^{1/2}\norm{\uvel}_{\AR(\torus;\LR{s_1}(\Omega))}
+\rey^{1/4}\norm{\grad\uvel}_{\AR(\torus;\LR{s_2}(\Omega))} 
+\norm{\grad\upres}_{\AR(\torus;\LR{q}(\Omega))}
\leq\const{const:RotatingOseen_linear}\norm{f}_{\AR(\torus;\LR{q}(\Omega))}
\end{split}
\end{align} 
for the constant
$\const{const:RotatingOseen_linear}$ from Theorem \ref{thm:RotatingOseen_resolvent},
and $s_1=2q/\np{2-q}$, $s_2=4q/\np{4-q}$.
Additionally, if $\np{\wvel,\wpres}$ 
is another solution to \eqref{sys:RotatingOseen_linear}
in the function class defined by the norms on the left-hand side of \eqref{est:RotatingOseen_linear},
then $\uvel=\wvel$ and $\upres=\wpres+\wpres_0$
for some (spatially constant) function $\wpres_0\colon\torus\to\R$.
\end{thm}

In Section \ref{sec:ProofNonlinear}, 
we finally prove the following existence result
on solutions to the nonlinear system 
\eqref{sys:RotatingOseen_Dimensionless}.

\begin{thm}\label{thm:RotatingOseen_nonlinear}
Let $\Omega\subset\R^3$ be an exterior domain of class $\CR{3}$, 
and let $q\in\bb{\frac{6}{5},\frac{4}{3}}$.
Let $f\in\AR(\torus;\LR{q}(\Omega))^3$ and $\veltrans\in\AR(\torus;\R)$ 
such that $\ddt\veltrans\in\AR(\torus;\R)$.
Define 
\[
\rey\coloneqq\frac{1}{2\pi}\int_0^{2\pi}\veltrans(t)\,\dt.
\]
For all $\rho\in\bp{\frac{3q-3}{q},1}$ and $\theta>0$ there are constants $\kappa>0$ and $\rey_0>0$ such that for all
\begin{align}\label{Mads_omega_lambda_scaling}
\rey\in (0,\rey_0),\qquad \tay\in\bp{\frac{\rey^2}{\theta},\kappa\rey^\rho}
\end{align}
there exists $\varepsilon>0$ such that if
\[
\norm{\veltrans-\rey}_{\AR(\torus;\R)} 
+\norm{f}_{\AR(\torus;\LR{q}(\Omega))}\leq\varepsilon,
\]
then there is a solution 
$\np{\uvel,\upres}$
to \eqref{sys:RotatingOseen_Dimensionless} with 
\begin{align*}
\uvel\in\AR(\torus;\LR{2q/(2-q)}(\Omega))^3, \quad
\grad\uvel\in\AR(\torus;\LR{4q/(4-q)}(\Omega))^{3\times3}, \quad
\grad^2\uvel&\in\AR(\torus;\LR{q}(\Omega))^{3\times3\times3},\\
\rotdertermsimple\uvel,\
\partial_1\uvel,\
\grad\upres
\in\AR(\torus;\LR{q}(\Omega))^3.&
\end{align*}
\end{thm}

\begin{rem}
The lower bound $\frac{\rey^2}{\theta}\leq\tay$ on the angular velocity in \eqref{Mads_omega_lambda_scaling} may seem strange in light of the underlying physics of the problem.
From a physical point of view, the limit $\tay\ra 0$ towards the case of a non-rotating body seems uncritical.
The lower bound on $\tay$ in \eqref{Mads_omega_lambda_scaling} is an artifact of the
change of coordinates into the rotating frame of reference employed in the mathematical analysis of the problem, which leads to \textit{a priori} estimates with constants exhibiting a singular behavior as $\tay\ra 0$. As a consequence, a lower bound on $\tay$ is required in Theorem \ref{thm:RotatingOseen_nonlinear} to obtain existence of a solution via a fixed-point iteration.  
A similar observation was made in the investigation of a steady flow past a rotating and translating obstacle carried out in \cite{FarwigHishidaMueller}.
From a mathematical point of view, it is therefore not surprising to see the same effect appearing in the more general time-periodic case investigated here. 
\end{rem}

\section{Preliminaries}\label{sec:Preliminaries}

We use capital letters to denote global constants, 
while constants in small letters are local to the respective proof. 
When we want to emphasize that a constant $C$ depends on the quantities $\alpha,\beta,\gamma,\dots$,
we write $C(\alpha,\beta,\gamma,\dots)$.

We denote points in $\torus\times\R^3$ by $(t,x)$, 
where $t$ and $x=\np{x_1,x_2,x_3}$ are referred to as time and spatial variable. 
The symbol $\Omega$ always denotes an exterior domain, that is,
$\Omega\subset\R^3$ is connected and the complement of a non-empty compact set. 
We always assume that the origin is not contained in $\Omega$.

Inner and outer product of two vectors $a,b\in\R^3$ are denoted 
by $a\cdot b$ and $a\wedge b$, respectively. 
For any radius $R>0$ we set  $\ball_R\coloneqq\setcl{x\in\R^3}{\snorm{x}<R}$, 
$\ball^R\coloneqq\setcl{x\in\R^3}{\snorm{x}>R}$,
and for a domain $D\subset\R^3$ we define
$D_R\coloneqq D\cap\ball_R$ and $D^R\coloneqq D\cap\ball^R$.

For $q\in[1,\infty]$, $k\in\N_0$, 
the symbols $\LR{q}(D)$ and $\WSR{k}{q}(D)$ denote usual Lebesgue and Sobolev spaces with
associated norms $\norm{\cdot}_{q}=\norm{\cdot}_{q;D}$ and 
$\norm{\cdot}_{k,q}=\norm{\cdot}_{k,q;D}$, respectively. 
Furthermore, $\WSRN{1}{q}(D)$ denotes the subset of functions in $\WSR{1}{q}(D)$
with vanishing boundary trace,
and $\WSR{-1}{q}(D)$ (with norm $\norm{\cdot}_{-1,q;D}$) 
is the dual space of $\WSRN{1}{q'}(D)$ where $1/q+1/q'=1$
with the usual convention $1/\infty\coloneqq 0$.
Moreover, $\LRsigma{2}(D)$ denotes the set of solenoidal vector fields in $\LR{2}(D)^3$, that is,
\[
\LRsigma{2}(D)\coloneqq\closure{\setcl{\varphi\in\CRci(D)^3}{\Div\varphi=0}}{\norm{\cdot}_{2}},
\] 
and $\projh$ is the corresponding Helmholtz projection that maps $\LR{2}(D)^3$ onto $\LRsigma{2}(D)$.

We always identify $2\pi$-periodic functions
with functions on the torus group $\torus\coloneqq\R/2\pi\Z$,
which is usually represented by the set $[0,2\pi)$.
We consider $\torus$ and $\grp\coloneqq\torus\times\R^3$ as locally compact abelian groups.
The (normalized) Haar measure on $\torus$ is given by
\[
\forall f\in\CR{}(\torus):\quad \int_\torus f\,\dt \coloneqq \frac{1}{2\pi}\int_0^{2\pi} f(t)\,\dt,
\]
and $\grp$ is equipped with the corresponding product measure.
Recall that the dual group of 
$\torus$ can be identified with $\dual{\torus}=\Z$
and that of
$\grp$ with $\dualgrp\coloneqq\Z\times\R^3$.

For $H=\torus$ or $H=\grp$, 
the space $\SR(H)$ is the Schwartz--Bruhat space of generalized Schwartz functions on $H$,
and $\TDR(H)$ denotes the corresponding dual space of tempered distributions;
see \cite{Bruhat61, EiterKyed_tplinNS_PiFbook} for precise definitions.
The Fourier transform on $\torus$ and $\grp$ and the respective inverses are given by
\begin{align*}
\FT_\torus\colon\SR(\torus)\ra\SR(\Z), &&\FT_\torus\nb{\uvel}(k)&\coloneqq
\int_\torus \uvel(t)\,\e^{-ik t}\,\dt,\\
\iFT_\torus\colon\SR(\Z)\ra\SR(\torus), &&\iFT_\torus\nb{\wvel}(t)&\coloneqq
\sum_{k\in\Z} \wvel(k)\,\e^{ik t},\\
\FT_\grp\colon\SR(\grp)\ra\SR(\dualgrp), &&\FT_\grp\nb{\uvel}(k,\xi)&\coloneqq
\int_\torus\int_{\R^n} \uvel(t,x)\,\e^{-ix\cdot\xi-ik t}\,\dx\dt,\\
\iFT_\grp\colon\SR(\dualgrp)\ra\SR(\grp), &&\iFT_\grp\nb{\wvel}(t,x)&\coloneqq
\sum_{k\in\Z}\,\int_{\R^n} \wvel(k,\xi)\,\e^{ix\cdot\xi+ik t}\,\dxi,
\end{align*}
provided the Lebesgue measure $\dxi$ is correctly normalized.
By duality, $\FT_\torus$ and $\FT_\grp$ are extended to homeomorphisms 
$\FT_\torus\colon\TDR(\torus)\ra\TDR(\Z)$ and $\FT_\grp\colon\TDR(\grp)\ra\TDR(\dualgrp)$, respectively.

Furthermore, we introduce 
the Sobolev space
\begin{align*}
&\WSR{1,2}{q}(\torus\times D) \coloneqq \closure{\CRci(\torus\times\overline{D})}{\norm{\cdot}_{1,2,q}},\quad 
\norm{f}_{1,2,q}\coloneqq\Bp{
\norm{\partial_t f}_{q}^q +
\sum_{k=0}^2 \norm{\grad^k f}_{q}^q  
}^{\frac{1}{q}},
\end{align*}
where $\CRci(\torus\times\overline{D})$ 
denotes the space of smooth functions of compact support on $\torus\times \overline{D}$ .

Let $X$ denote a Banach space.
For functions $\uvel\in\LR{1}(\torus;X)$
we introduce the projections $\proj$ and $\projcompl$ by 
\[
\proj\uvel
\coloneqq\int_{\torus}\uvel(t)\,\dt, \qquad
\projcompl\coloneqq\id-\proj.
\]
Note that $\proj\uvel\in X$ is time-independent, 
and we have the decomposition $\uvel=\proj\uvel+\projcompl\uvel$
into the \emph{steady-state} part $\proj\uvel$ 
and the \emph{purely periodic} part $\projcompl\uvel$ of $\uvel$.

Our analysis of the time-periodic problems 
\eqref{sys:RotatingOseen_Dimensionless} and \eqref{sys:RotatingOseen_linear} will be carried out 
within spaces of functions with absolutely convergent Fourier series defined by
\begin{align}\label{eq:DefOfA}
\begin{aligned}
\AR(\torus;X)
&\coloneqq
\setcL{f\colon\torus\to X}{f(t)=\sum_{k\in\Z}f_k \e^{ikt}, \ f_k\in X, \ 
\sum_{k\in\Z}\norm{f_k}_{X}<\infty},
\\
\norm{f}_{\AR(\torus;X)}
&\coloneqq\sum_{k\in\Z}\norm{f_k}_{X}.
\end{aligned}
\end{align}
Observe that $\AR(\torus;X)$ is the Banach space that coincides with 
$\iFT_\torus\bb{\lR{1}(\Z;X)}$, which embeds into the $X$-valued
continuous functions on $\torus$.
It is well known that the scalar-valued space $\AR(\torus;\R)$ is an algebra 
with respect to pointwise multiplication, the so-called Wiener algebra. 
One can exploit this property to derive estimates in the $X$-valued case.
For example, one readily shows the following correspondences of H\"older's inequality 
and interpolation inequalities.

\begin{prop}\label{prop:Hoelder_AR}
Let $D\subset\R^n$, $n\in\N$, be an open set and $p,q,r\in[1,\infty]$ such that $1/p+1/q=1/r$.
Moreover, let $f\in\AR(\torus;\LR{p}(D))$ and $g\in\AR(\torus;\LR{q}(D))$.
Then $fg\in\AR(\torus;\LR{r}(D))$ and
\begin{align}\label{est:Hoelder_AR}
\norm{fg}_{\AR(\torus;\LR{r}(D))}
\leq\norm{f}_{\AR(\torus;\LR{p}(D))}
\norm{g}_{\AR(\torus;\LR{q}(D))}.
\end{align}
\end{prop}
\begin{proof}
By assumption we have $f=\iFT_\torus\nb{(f_k)}$ and $g=\iFT_\torus\nb{(g_k)}$
for elements $(f_k)\in\lR{1}(\Z;\LR{p}(D))$ and $(g_k)\in\lR{1}(\Z;\LR{q}(D))$.
Then $fg=\iFT_\torus\bb{(f_k)\ast_\Z (g_k)}$ and
\begin{align*}
\norm{fg}_{\AR(\torus;\LR{r}(D))}
&=\sum_{k\in\Z} \normL{\sum_{\ell\in\Z} f_\ell g_{k-\ell}}_{\LR{r}(D)}
\leq\sum_{k\in\Z}\sum_{\ell\in\Z}\norm{f_\ell g_{k-\ell}}_{\LR{r}(D)}\\
&\leq\sum_{k\in\Z}\sum_{\ell\in\Z} \norm{f_\ell}_{\LR{p}(D)}\norm{g_{k-\ell}}_{\LR{q}(D)}
=\norm{f}_{\AR(\torus;\LR{p}(D))}
\norm{g}_{\AR(\torus;\LR{q}(D))},
\end{align*}
where the last estimate is due to H\"older's inequality.
\end{proof}

\begin{prop}\label{prop:Interpolation_AR}
Let $D\subset\R^n$, $n\in\N$, be an open set and $p,q,r\in[1,\infty]$ such that
$(1-\theta)/p+\theta/q=1/r$ for some $\theta\in[0,1]$,
and let $f\in\AR(\torus;\LR{p}(D))\cap\AR(\torus;\LR{q}(D))$.
Then $f\in\AR(\torus;\LR{r}(D))$ and
\begin{align}\label{est:Interpolation_AR}
\norm{f}_{\AR(\torus;\LR{r}(D))}
\leq\norm{f}_{\AR(\torus;\LR{p}(D))}^{1-\theta}
\norm{f}_{\AR(\torus;\LR{q}(D))}^\theta.
\end{align}
\end{prop}
\begin{proof}
We have $f=\iFT_\torus\nb{(f_k)}$
for an element $(f_k)\in\lR{1}(\Z;\LR{p}(D)\cap\LR{q}(D))$.
The classical interpolation inequality for Lebesgue spaces yields
\[
\norm{f}_{\AR(\torus;\LR{r}(D))}
=\sum_{k\in\Z} \norm{f_k}_{\LR{r}(D)}
\leq\sum_{k\in\Z}\norm{f_k}_{\LR{p}(D)}^{1-\theta}
\norm{f_k}_{\LR{q}(D)}^\theta
\leq\norm{f}_{\AR(\torus;\LR{p}(D))}^{1-\theta}
\norm{f}_{\AR(\torus;\LR{q}(D))}^\theta,
\]
where the last estimate follows from H\"older's inequality on $\Z$.
\end{proof}

\section{Embedding theorem}\label{sec:EmbeddingThm}

This section deals with embedding properties of Sobolev spaces of time-periodic functions.
The embedding theorem below is a refinement of \cite[Theorem 4.1]{GaldiKyed_TPflowViscLiquidpBody_2018}
adapted to the time-scaling employed in \eqref{sys:RotatingOseen_Dimensionless}.
Clearly, embeddings of the steady-state part $\proj\uvel$ are independent of the actual period. 
Therefore, we only consider the case of purely periodic functions.
For the sake of generality, we establish the following theorem in arbitrary dimension $n\geq 2$.

\begin{thm}\label{thm:EmbeddingThm_WholeSpace}
Let $n\geq 2$, $\tay>0$ and $q\in(1,\infty)$.
For $\alpha\in[0,2]$ with $\alpha q<2$ and $(2-\alpha) q <n$ let
\[
r_0\coloneqq \frac{2q}{2-\alpha q}, \qquad 
p_0\coloneqq \frac{nq}{n-(2-\alpha)q}, 
\]
and for $\beta\in[0,1]$ with $\beta q<2$ and $(1-\beta) q<n$ let
\[
r_1\coloneqq \frac{2q}{2-\beta q}, \qquad
p_1\coloneqq \frac{nq}{n-(1-\beta)q}.
\]
Then the inequality
\begin{align}\label{est:EmbeddingThm_WholeSpace}
\tay^{\alpha/2}\norm{\uvel}_{\LR{r_0}(\torus;\LR{p_0}(\Rn))}
+\tay^{\beta/2}\norm{\grad\uvel}_{\LR{r_1}(\torus;\LR{p_1}(\Rn))}
&\leq \Cc[const:EmbeddingTheorem_WholeSpace]{C} 
\bp{
\tay\norm{\partial_t\uvel}_q
+\norm{\grad^2\uvel}_q
}
\end{align}
holds
for all $\uvel\in\projcompl\WSR{1,2}{q}(\torus\times\Rn)$
and a constant 
$\const{const:EmbeddingTheorem_WholeSpace}=\const{const:EmbeddingTheorem_WholeSpace}(n,q,\alpha,\beta)>0$.
\end{thm}

\begin{proof}
Since the proof is analogue to 
\cite[Proof of Theorem 4.1]{GaldiKyed_TPflowViscLiquidpBody_2018},
we merely give a brief sketch here.
Without restriction we may assume $\uvel\in\SR(\grp)$.
Due to the assumption $\uvel=\projcompl\uvel$, 
we have $\FT_\grp\nb{\uvel}=\np{1-\delta_\Z}\FT_\grp\nb{\uvel}$,
where $\delta_\Z$ is the delta distribution on $\Z$.
Utilizing the Fourier transform, 
we thus derive the identity
\begin{align}\label{eq:EmbeddingThm_WholeSpace_uRepresentation}
\begin{split}
\uvel
&=\iFT_\grp\Bb{
\frac{1-\delta_\Z(k)}{\snorm{\xi}^2+i\tay k} \FT_G\bb{\tay\partial_t\uvel-\Delta\uvel}
} \\
&=\tay^{-\alpha/2}
\iFT_{\Rn}\bb{
\snorm{\xi}^{\alpha-2}
}
\ast_{\Rn}
\iFT_{\torus}\bb{
\np{1-\delta_\Z}\snorm{k}^{-\alpha/2}
}
\ast_{\torus}
F,
\end{split}
\end{align}
where 
\[
F\coloneqq
\iFT_\grp \Bb{M_\tay(k,\xi)\FT_G\bb{\tay\partial_t\uvel-\Delta\uvel}},
\qquad
M_\tay(k,\xi)
\coloneqq\frac
{\snorm{\tay k}^{\alpha/2}\snorm{\xi}^{2-\alpha}\np{1-\delta_\Z(k)}}
{\snorm{\xi}^2+i\tay k}.
\]
Employing the so-called transference principle for Fourier multipliers 
(see \cite{EdwardsGaudryBook,EiterKyed_tplinNS_PiFbook})
together with the Marcinkiewicz multiplier theorem,
one readily verifies that $M_\tay$ is an $\LR{q}(\grp)$ multiplier for any $q\in(1,\infty)$
such that
\[
\norm{F}_{q}\leq\Cc[const:EmbeddingTheorem_WholeSpace:Mmultiplier]{c}\norm{\tay\partial_t\uvel-\Delta\uvel}_{q}
\leq\const{const:EmbeddingTheorem_WholeSpace:Mmultiplier}\bp{\tay\norm{\partial_t\uvel}_q+\norm{\grad^2\uvel}_{q}}
\]
with $\const{const:EmbeddingTheorem_WholeSpace:Mmultiplier}$ independent of $\tay$.
Moreover, when we chose $[-\pi,\pi)$ as a realization of $\torus$,
we obtain
\[
\gamma_\alpha(t)\coloneqq
\iFT_{\torus}\bb{
\np{1-\delta_\Z}\snorm{k}^{-\alpha/2}
}(t)
=\Cc{c}t^{-1+\alpha/2}+h(t),
\]
for some $h\in\CRi(\torus)$; see for example \cite[Example 3.1.19]{Grafakos1}.
In particular, this yields $\gamma_\alpha\in\LR{\frac{1}{1-\alpha/2},\infty}(\torus)$,
so that Young's inequality implies that the mapping $\varphi\mapsto\gamma_\alpha\ast\varphi$
extends to a bounded operator $\LR{q}(\torus)\to\LR{r_0}(\torus)$.
Moreover, it is well known that the mapping
$\varphi\mapsto\iFT_{\Rn}\bb{\snorm{\xi}^{\alpha-2}}\ast\varphi$ extends to a bounded operator
$\LR{q}(\Rn)\to\LR{p_0}(\Rn)$; see \cite[Theorem 6.1.13]{Grafakos2}.
Recalling \eqref{eq:EmbeddingThm_WholeSpace_uRepresentation}, we thus have
\begin{align*}
\tay^{\alpha/2}\norm{\uvel}_{\LR{r_0}(\torus;\LR{p_0}(\Rn))}
&=\Bp{\int_{\torus}\normL{\iFT_{\Rn}\bb{
\snorm{\xi}^{\alpha-2}
}
\ast_{\Rn}
\gamma_\alpha
\ast_{\torus}
F(t,\cdot)}_{p_0}^{r_0}\,\dt}^{\frac{1}{r_0}} \\
&\leq\Cc{c}\Bp{\int_{\torus}\norm{
\gamma_\alpha
\ast_{\torus}
F(t,\cdot)}_{q}^{r_0}\,\dt}^{\frac{1}{r_0}} 
\leq\Cc{c}\Bp{\int_{\Rn}\norm{
\gamma_\alpha
\ast_{\torus}
F(\cdot,x)}_{r_0}^{q}\,\dx}^{\frac{1}{q}} \\
&\leq\Cc{c}\norm{F}_{q}
\leq\Cc{c}\bp{\tay\norm{\partial_t\uvel}_{q}+\norm{\grad^2\uvel}_{q}},
\end{align*}
where Minkowski's integral inequality is used in the second estimate. 
This is the asserted inequality for $\uvel$.
The estimate of $\grad\uvel$ follows in the same way.
\end{proof}

\begin{rem}\label{rem:EquivNormTay_WholeSpace}
Note that the term on the right-hand side of \eqref{est:EmbeddingThm_WholeSpace} defines
a norm equivalent to $\norm{\cdot}_{1,2,q}$ on $\projcompl\WSR{1,2}{q}(\torus\times\Omega)$
due to Poincar\'e's inequality on $\torus$.
\end{rem}

\begin{rem}\label{rem:EmbeddingTheorem_ExtDomain}
Theorem \ref{thm:EmbeddingThm_WholeSpace} can be generalized to the setting of an exterior domain
$\Omega\subset\Rn$ by means of Sobolev extensions.
However, to maintain estimate \eqref{est:EmbeddingThm_WholeSpace}, 
one has to construct a specific extension operator that respects the homogeneous 
second-order Sobolev norm. 
To this end, one can make use of results from \cite{Burenkov_ExtensionFunctionsPreservationSeminorms}.
\end{rem}

\section{Linear theory}\label{sec:linear}
This section is dedicated to the investigation of the 
resolvent problem \eqref{sys:RotatingOseen_resolvent} and the linear time-periodic problem \eqref{sys:RotatingOseen_linear}. 
After having shown Theorem \ref{thm:RotatingOseen_resolvent},
we establish Theorem \ref{thm:RotatingOseen_linear} as an immediate consequence hereof.

\subsection{The whole space}
To study the problems \eqref{sys:RotatingOseen_linear} and
\eqref{sys:RotatingOseen_resolvent} in an exterior domain,
we first consider the case $\Omega=\R^3$.
In the whole-space setting one can namely change coordinates back to the non-rotating inertial frame and thereby reduce
the study of \eqref{sys:RotatingOseen_linear} to an investigation of 
the time-periodic Oseen problem without rotation
terms, which was analyzed in 
\cite{Kyed_mrtpns,GaldiKyed_TPflowViscLiquidpBody_2018}. 
In this section, we set
\[
s_1\coloneqq\frac{2q}{2-q}, \qquad
s_2\coloneqq\frac{4q}{4-q}, \qquad
s_3\coloneqq\frac{8q}{8-q}.
\]
for appropriately fixed $q$.

\begin{thm}\label{thm:Oseen_WholeSpace}
Let $q\in(1,2)$ and $\rey,\,\tay,\,\theta>0$ with $\rey^2\leq\theta\tay$.
For every $f\in\LR{q}(\torus\times\R^3)^3$
there exists a solution $\np{\uvel,\upres}\in\TDR(\torus\times\R^3)^{3+1}$ to 
\begin{align}\label{sys:Oseen_WholeSpace_Tay}
\begin{pdeq}
\tay\partial_t \uvel- \Delta \uvel - \rey \partial_1 \uvel + \grad \upres &= f
& \tin \torus\times\R^3, \\
\Div\uvel&=0
& \tin \torus\times\R^3,
\end{pdeq}
\end{align}
with
$\partial_t\uvel,\grad^2\uvel,\,\grad\upres\in\LR{q}(\torus\times\R^3)$.
Moreover, there exist constants
$\Cc[const:Oseen_WholeSpace_steadystate]{C}=\const{const:Oseen_WholeSpace_steadystate}\np{q}>0$
and
$\Cc[const:Oseen_WholeSpace_purelyperiodic]{C}=\const{const:Oseen_WholeSpace_purelyperiodic}\np{q,\theta}>0$
such that
\begin{align}
\norm{\grad^2\proj\uvel}_{q}
+\rey\norm{\partial_1\proj\uvel}_{q}
+\rey^{1/2}\norm{\proj\uvel}_{s_1}
+\rey^{1/4}\norm{\grad\proj\uvel}_{s_2}
+\norm{\grad\proj\upres}_{q}
&\leq\const{const:Oseen_WholeSpace_steadystate}\norm{\proj f}_{q},
\label{est:Oseen_WholeSpace_steadystate}
\\
\tay\norm{\partial_t\projcompl\uvel}_{q}
+\norm{\grad^2\projcompl\uvel}_{q}
+\rey\norm{\partial_1\projcompl\uvel}_{q}
+\norm{\grad\projcompl\upres}_{q}
&\leq\const{const:Oseen_WholeSpace_purelyperiodic}\norm{\projcompl f}_{q}.
\label{est:Oseen_WholeSpace_purelyperiodic}
\end{align}
Additionally, if $\np{\wvel,\wpres}\in\TDR(\torus\times\R^3)^{3+1}$ is another solution to 
\eqref{sys:Oseen_WholeSpace_Tay}, 
then $\projcompl\uvel=\projcompl\wvel$, 
and $\proj\uvel-\proj\wvel$ is a polynomial in each component,
and $\upres-\wpres=\upres_0$, where $\upres_0(t,\cdot)$ is a polynomial for each $t\in\torus$.
\end{thm}

\begin{proof}
We decompose \eqref{sys:Oseen_WholeSpace_Tay} into two problems by splitting
$\uvel=\proj\uvel+\projcompl\uvel\eqqcolon\uvels+\uvelp$ 
and $\upres=\proj\upres+\projcompl\upres\eqqcolon\upress+\upresp$.
For the steady-state part $(\uvels,\upress)$ we obtain the system
\[
\begin{pdeq}
- \Delta \uvels - \rey \partial_1 \uvels + \grad \upress &= \proj f
& \tin \R^3, \\
\Div\uvels&=0
& \tin \R^3,
\end{pdeq}
\]
which is the classical steady Oseen problem.
The existence of a time-independent solution $\np{\uvels,\upress}$ satisfying estimate \eqref{est:Oseen_WholeSpace_steadystate} 
is well known; see for example \cite[Theorem VII.4.1]{GaldiBookNew}. 
The remaining purely periodic part $\np{\uvelp,\upresp}$ must solve \eqref{sys:Oseen_WholeSpace_Tay},
but with purely periodic right-hand side $\projcompl f$. 
We define
\[
\Uvel(t,x)\coloneqq\,\uvelp(t, \tay^{-1/2} x), \quad
\Upres(t,x)\coloneqq\tay^{-1/2}\,\upresp(t, \tay^{-1/2}x), \quad
F(t,x)\coloneqq \tay^{-1}\,\projcompl f(t, \tay^{-1/2}x),
\]
which leads to the system
\[
\begin{pdeq}
\partial_t \Uvel- \Delta \Uvel - \widetilde{\rey} \partial_1 \Uvel + \grad \Upres &= F
& \tin \torus\times\R^3, \\
\Div\Uvel&=0
& \tin \torus\times\R^3,
\end{pdeq}
\]
where $\widetilde{\rey}=\rey\tay^{-1/2}$.
From \cite[Theorem 2.1]{Kyed_mrtpns} we conclude the existence of 
a unique solution $\np{\Uvel,\Upres}$ that satisfies
the estimate 
\[
\norm{\Uvel}_{1,2,q}+\norm{\grad\Upres}_{q}
\leq\Cc[const:Oseen_WholeSpace_purelyperiodicRescaled]{c}\norm{F}_{q},
\]
where $\const{const:Oseen_WholeSpace_purelyperiodicRescaled}$ is a polynomial in $\widetilde{\rey}$ 
and can thus be bounded uniformly in $\widetilde{\rey}\in(0,\sqrt{\theta}]$.
Estimate \eqref{est:Oseen_WholeSpace_purelyperiodic}
with the asserted dependency of the constant $\const{const:Oseen_WholeSpace_purelyperiodic}$ 
follows after reversing the applied scaling. 

The uniqueness statement is readily shown by means of
the Fourier transform on $\grp=\torus\times\R^3$. 
We consider \eqref{sys:Oseen_WholeSpace_Tay} with $f=0$ and apply
the divergence operator to \eqrefsub{sys:Oseen_WholeSpace_Tay}{1}. 
This yields $\Delta\upres=0$ 
and thus $\snorm{\xi}^2\FT_{\R^3}\nb{\upres(t,\cdot)}=0$ for all $t\in\torus$.
Therefore, we obtain $\supp \FT_{\R^3}\nb{\upres(t,\cdot)} \subset \set{0}$, so that 
$\upres(t,\cdot)$ is a polynomial for all $t\in\torus$.
Next we apply the Fourier transform to \eqrefsub{sys:Oseen_WholeSpace_Tay}{1} 
to deduce
$\np{i\tay k+\snorm{\xi}^2-i\xi_1}\FT_G\nb{\uvel}+i\xi\FT_\grp\nb{\upres}=0$.
Multiplying with the symbol of the Helmholtz projection $\idmatrix-\xi\otimes\xi/\snorm{\xi}^2$
and utilizing $\Div\uvel=0$, we obtain
$\np{i\tay k+\snorm{\xi}^2-i\xi_1}\FT_G\nb{\uvel}=0$,
which yields $\supp \FT_G\bb{\uvel}\subset\{(0,0)\}$.
Since ${\projcompl\uvel}=\iFT_\grp\bb{\np{1-\delta_\Z}\FT_\grp\nb{\uvel}}$,
it follows that $\projcompl\uvel=0$, and that each component of $\proj\uvel$ is a polynomial.
This completes the proof.
\end{proof}

\begin{rem}\label{rem:TPOseen_WholeSpace_EquivNorm}
In the setting of Theorem \ref{thm:Oseen_WholeSpace}
we can write the estimate for the steady-state part 
$\np{\uvels,\upress}=\np{\proj\uvel,\proj\upres}$ 
and the purely periodic part 
$\np{\uvelp,\upresp}=\np{\projcompl\uvel,\projcompl\upres}$ 
in a more condensed way:
From the embeddings established in Theorem \ref{thm:EmbeddingThm_WholeSpace}
we deduce
\begin{align*}
\tay^{1/4}\norm{\uvelp}_{\LR{s_2}(\torus;\LR{s_1}(\R^3))}
+\tay^{1/8}\norm{\grad\uvelp}_{\LR{s_3}(\torus;\LR{s_2}(\R^3))}
\leq\Cc{C}\bp{\tay\norm{\partial_t\uvelp}_{\LR{q}(\torus\times\R^3)}
+\norm{\uvelp}_{\LR{q}(\torus\times\R^3)}}.
\end{align*}
Recalling Remark \ref{rem:EquivNormTay_WholeSpace},
we see that 
\eqref{est:Oseen_WholeSpace_steadystate} and \eqref{est:Oseen_WholeSpace_purelyperiodic}
can be formulated as
\begin{multline}\label{est:Oseen_WholeSpace_combined}
\tay\norm{\partial_t\uvel}_{q}
+\norm{\grad^2\uvel}_{q}
+\rey\norm{\partial_1\uvel}_{q} 
+\rey^{1/2}\norm{\uvel}_{\LR{s_2}(\torus;\LR{s_1}(\R^3))}\\
+\rey^{1/4}\norm{\grad\uvel}_{\LR{s_3}(\torus;\LR{s_2}(\R^3))} 
+\norm{\grad\upres}_{q}
\leq\Cc[const:Oseen_WholeSpace_combined]{C}
\norm{f}_{q}
\end{multline}
for a constant $\const{const:Oseen_WholeSpace_combined}=\const{const:Oseen_WholeSpace_combined}(q,\theta)$ 
as long as $\rey^2\leq\theta\tay$. 
\end{rem}
With Theorem \ref{thm:Oseen_WholeSpace} 
we now solve the linear problem \eqref{sys:RotatingOseen_linear}
for $\Omega=\R^3$ and $f\in\LR{q}(\torus\times\R^3)^3$.

\begin{thm}\label{thm:RotatingOseen_WholeSpace}
Let $q\in(1,2)$ and $\rey,\,\tay, \,\theta>0$ with $\rey^2\leq\theta\tay$.
For every $f\in\LR{q}(\torus\times\R^3)^3$
there exists a solution $\np{\uvel,\upres}\in\TDR(\torus\times\R^3)^{3+1}$ to 
\begin{align}\label{sys:RotatingOseen_WholeSpace}
\begin{pdeq}
\tay\np{\partial_t \uvel+\eone\wedge\uvel-\eone\wedge x\cdot\grad\uvel}
-\Delta \uvel - \rey \partial_1 \uvel + \grad \upres &= f
& \tin \torus\times\R^3, \\
\Div\uvel&=0
& \tin \torus\times\R^3,
\end{pdeq}
\end{align}
with
$\grad^2\uvel,\,\partial_1\uvel,\,\grad\upres\in\LR{q}(\torus\times\R^3)$.
Moreover, there exists a constant
$\Cc[const:RotatingOseen_WholeSpace]{C}
=\const{const:RotatingOseen_WholeSpace}\np{q,\theta}>0$
such that
\begin{align}\label{est:RotatingOseen_WholeSpace}
\begin{split}
\tay&\norm{\rotdertermsimple{\uvel}}_{\LR{q}(\torus\times\R^3)}
+\norm{\grad^2\uvel}_{\LR{q}(\torus\times\R^3)}
+\rey\norm{\partial_1\uvel}_{\LR{q}(\torus\times\R^3)} 
\\
&+\rey^{1/2}\norm{\uvel}_{\LR{s_2}(\torus;\LR{s_1}(\R^3))}
+\rey^{1/4}\norm{\grad\uvel}_{\LR{s_3}(\torus;\LR{s_2}(\R^3))} 
+\norm{\grad\upres}_{\LR{q}(\torus\times\R^3)}
\leq\const{const:RotatingOseen_WholeSpace}
\norm{f}_{\LR{q}(\torus\times\R^3)}.
\end{split}
\end{align}
Additionally, if $\np{\wvel,\wpres}\in\TDR(\torus\times\R^3)^{3+1}$ is another solution to 
\eqref{sys:RotatingOseen_WholeSpace}
with $\wvel\in\LR{r}(\torus\times\R^3)$ for some $r\in[1,\infty)$, 
then $\uvel=\wvel$, and $\upres-\wpres=\wpres_0$ for some spatially constant 
function $\wpres_0\colon\torus\to\R$.
\end{thm}
\begin{proof}
Let
\[
\rotmatrix(t)\coloneqq
\begin{pmatrix}
1 & 0 & 0 \\
0 & \cos(t) & -\sin(t) \\
0 & \sin(t) & \cos(t)
\end{pmatrix}
\]
be the matrix corresponding to the rotation with angular velocity $\eone$. 
Define 
\[
\begin{aligned}
\Uvel(t,y)&\coloneqq\rotmatrix(t)\uvel(t,\rotmatrix(t)^\transpose y),
\\
\Upres(t,y)&\coloneqq\upres(t,\rotmatrix(t)^\transpose y),
\\
F(t,y)&\coloneqq\rotmatrix(t) f(t,\rotmatrix(t)^\transpose y).
\end{aligned}
\]
with the new spatial variable $y=\rotmatrix(t)x$.
Due to 
\[
\partial_t\Uvel(t,y)=\rotmatrix(t) \rotderterm{\uvel(t,x)},
\]
the functions $\uvel$, $\upres$ and $f$ satisfy \eqref{sys:RotatingOseen_WholeSpace} 
if and only if
\[
\begin{pdeq}
\tay\partial_t \Uvel- \Delta \Uvel - \rey \partial_1 \Uvel + \grad \Upres &= F
& \tin \torus\times\R^3, \\
\Div\Uvel&=0
& \tin \torus\times\R^3.
\end{pdeq}
\]
The assertions in Theorem \ref{thm:RotatingOseen_WholeSpace} are now a direct consequence of Theorem \ref{thm:Oseen_WholeSpace}
and estimate \eqref{est:Oseen_WholeSpace_combined}.
\end{proof}

\begin{rem}
As for the corresponding steady-state problem (see for example \cite[Theorem VIII.8.1]{GaldiBookNew}), 
one can extend Theorem \ref{thm:RotatingOseen_WholeSpace} to the case of an exterior domain $\Omega$ 
for $f\in\LR{q}(\torus\times\Omega)$,
but it is not clear to the authors 
whether or not the constant in the resulting \textit{a priori} estimate 
can then be chosen independently of $\rey$ and $\tay$.
Observe that such an independence is obtained in the functional setting of Theorem \ref{thm:RotatingOseen_linear} 
where $f\in\AR(\torus;\LR{q}(\Omega))$. 
Since we solve the nonlinear problem \eqref{sys:RotatingOseen_Dimensionless} via a fixed-point iteration which requires $\rey$ and $\tay$
to be chosen sufficiently small, it crucial to obtain an estimate with the constant independent of $\rey$ and $\tay$.
\end{rem}

From Theorem \ref{thm:RotatingOseen_WholeSpace}
we can extract a similar result for the resolvent problem
\eqref{sys:RotatingOseen_resolvent} in the whole space.

\begin{thm}\label{thm:RotatingOseen_WholeSpace_resolvent}
Let $q\in(1,2)$, $k\in\Z$ and $\rey,\,\tay, \,\theta>0$ with $\rey^2\leq\theta\tay$.
For every $\vf\in\LR{q}(\R^3)^3$
there exists a solution $\np{\vvel,\vpres}\in\TDR(\R^3)^{3+1}$ to 
\begin{align}\label{sys:RotatingOseen_WholeSpace_resolvent}
\begin{pdeq}
\tay\rotdertermFT{\vvel}
-\Delta \vvel - \rey \partial_1 \vvel + \grad \vpres &= \vf
& \tin\R^3, \\
\Div\vvel&=0
& \tin\R^3,
\end{pdeq}
\end{align}
and a constant
$\Cc[const:RotatingOseen_WholeSpace_resolvent]{C}
=\const{const:RotatingOseen_WholeSpace_resolvent}\np{q,\theta}>0$
with
\begin{align}\label{est:RotatingOseen_WholeSpace_resolvent}
\begin{split}
\tay&\norm{\rotdertermFTsimple{\vvel}}_{q}
+\norm{\grad^2\vvel}_{q}
+\rey\norm{\partial_1\vvel}_{q}
\\
&\qquad \qquad \qquad\qquad+\rey^{1/2}\norm{\vvel}_{s_1}
+\rey^{1/4}\norm{\grad\vvel}_{s_2} 
+\norm{\grad\vpres}_{q}
\leq\const{const:RotatingOseen_WholeSpace_resolvent}
\norm{\vf}_{q}.
\end{split}
\end{align}
Additionally, if $\np{\wvel,\wpres}\in\SR(\R^3)^{3+1}$ 
is another solution to 
\eqref{sys:Oseen_WholeSpace_Tay}
with $\wvel\in\LR{r}(\Omega)$ for some $r\in[1,\infty)$,
then $\vvel=\wvel$, and $\vpres-\wpres$ is constant.
\end{thm}
\begin{proof}
First consider a solution  $(\vvel,\vpres)$ in the described function class.
Then the fields 
\[
\uvel(t,x)\coloneqq\e^{ikt}\vvel(x),
\qquad
\upres(t,x)\coloneqq\e^{ikt}\vpres(x),
\qquad
f(t,x)\coloneqq\e^{ikt}\vf(x),
\]
satisfy \eqref{sys:RotatingOseen_WholeSpace}.
Therefore, uniqueness of $\np{\vvel,\grad\vpres}$ follows from the uniqueness statement 
in Theorem \ref{thm:RotatingOseen_WholeSpace}.
To show existence, let $\vf\in\LR{q}(\R^3)$ and define $f\in\LR{q}(\torus\times\R^3)$ as above. 
Theorem \ref{thm:RotatingOseen_WholeSpace} yields the existence 
of a pair $(\uvel,\upres)$ that solves \eqref{sys:RotatingOseen_WholeSpace}. 
Then the $k$-th Fourier coefficients $\vvel(x)\coloneqq\FT_\torus\nb{\uvel(\cdot,x)}(k)$ and
$\vpres(x)\coloneqq\FT_\torus\nb{\upres(\cdot,x)}(k)$ 
satisfy \eqref{sys:RotatingOseen_WholeSpace_resolvent},
and estimate \eqref{est:RotatingOseen_WholeSpace_resolvent}
is a direct consequence of 
\eqref{est:RotatingOseen_WholeSpace}.
\end{proof}

\subsection{Uniqueness}

Next we show a uniqueness result for the resolvent problem \eqref{sys:RotatingOseen_resolvent}.

\begin{lem}\label{lem:RotatingOseen_Uniqueness_resolvent}
Let $\rey\geq 0$, $\tay>0$, $k\in\Z$, 
and let $\np{\vvel,\vpres}$ be a distributional solution to \eqref{sys:RotatingOseen_resolvent}
with $\vf=0$ and
$\grad^2\vvel,\,\partial_1\vvel,\,\grad\vpres \in\LR{q}(\Omega)$
for some $q\in(1,\infty)$ and $\vvel\in\LR{s}(\Omega)$ for some $s\in(1,\infty)$.
Then 
$\vvel=0$ and $\vpres$ is constant.
\end{lem}

\begin{proof}
We only consider the case $\rey>0$ here. 
The proof for $\rey=0$ can be shown in exactly the same way.
Fix a radius $R>0$ such that $\partial\ball_R\subset\Omega$, 
and define a ``cut-off'' function 
$\cutoff_0\in\CRci(\R^3)$ with 
$\cutoff_0(x)=1$ for $\snorm{x}\leq 2R$ and $\cutoff_0(x)=0$ for $\snorm{x}\geq 4R$. 
Set
\begin{align}\label{eq:RotatingOseen_Cutoff_InnerFunction}
\wvel\coloneqq\cutoff_0\vvel-\bogopr\np{\vvel\cdot\grad\cutoff_0}, 
\qquad
\wpres\coloneqq\cutoff_0\vpres
\end{align}
where $\bogopr$ denotes the Bogovski\u\i{} operator; see for example \cite[Section III.3]{GaldiBookNew}.
Then 
\[
\begin{pdeq}
- \Delta \wvel + \grad \wpres &= \wf
&& \tin \Omega_{4R}, \\
\Div\wvel&=0
&& \tin \Omega_{4R}, \\
\wvel&=0
&& \ton \partial\Omega_{4R},
\end{pdeq}
\]
with
\[
\wf\coloneqq
\bp{-\tay\rotdertermFT{\vvel}-\rey\partial_1\vvel}\cutoff_0
-2\grad\cutoff_0\cdot\grad\vvel-\Delta\cutoff_0\vvel+\grad\cutoff_0\vpres
+\Delta\bogopr(\grad\cutoff_0\cdot\vvel).
\]
From the assumptions, we obtain $\vvel\in\WSR{2}{q}(\Omega_{4R})$
and $\vpres\in\WSR{1}{q}(\Omega_{4R})$.  
Standard Sobolev embeddings 
imply $\vvel, \grad\vvel, \vpres\in\LR{\frac32 q}(\Omega_{4R})$.
Therefore, we also have $\wf\in\LR{r}(\Omega_{4R})$ 
for all $1< r\leq\frac{3}{2}q$.
From well-known regularity results for the Stokes problem in bounded domains 
(see \cite[Theorem IV.6.1]{GaldiBookNew})
we obtain $\wvel\in\WSR{2}{r}(\Omega_{4R})$ 
and $\grad\wpres\in\LR{r}(\Omega_{4R})$.
Since $\vvel=\wvel$ and $\vpres=\wpres$ on $\Omega_{2R}$,
this yields 
\begin{align}\label{el:RotatingOseen_Regularity_InnerFunction}
\np{\vvel,\vpres}\in
\WSR{2}{r}(\Omega_{2R})\times\WSR{1}{r}(\Omega_{2R})
\end{align}
for all $1< r\leq\frac{3}{2}q$.

Next consider another ``cut-off'' function $\cutoff_1\in\CRi(\R^3)$ with 
$\cutoff_1(x)=1$ for $\snorm{x}\geq 2R$ and $\cutoff_1(x)=0$ for $\snorm{x}\leq R$.
As above, we define 
\begin{align}\label{eq:RotatingOseen_Cutoff_OuterFunction}
\uvel\coloneqq\cutoff_1\vvel-\bogopr\np{\vvel\cdot\grad\cutoff_1},
\qquad
\upres\coloneqq\cutoff_1\vpres,
\end{align}
which satisfy the system
\begin{align}\label{sys:RotatingOseen_OuterFunction}
\begin{pdeq}
\tay\rotdertermFT{\uvel} -\Delta \uvel -\rey\partial_1\uvel + \grad \upres &= \uf
& \tin \R^3, \\
\Div\uvel&=0
& \tin \R^3,
\end{pdeq}
\end{align}
with
\begin{align*}
\uf&\coloneqq
\tay\np{\eone\wedge x\cdot\grad\cutoff_1}\vvel
-2\grad\cutoff_1\cdot\grad\vvel-\Delta\cutoff_1\vvel
+\rey\partial_1\cutoff_1\vvel
+\grad\cutoff_1\vpres -\Delta\bogopr(\vvel\cdot\grad\cutoff_1) 
\\
&
\quad+\rey\partial_1\bogopr(\vvel\cdot\grad\cutoff_1)
+\tay\rotdertermFT{\bogopr(\vvel\cdot\grad\cutoff_1)}.
\end{align*}
As above, we see $\uf\in\LR{r}(\R^3)$ for all $1< r\leq \frac{3}{2}q$.
Since we also have $\uvel\in\LR{s}(\R^3)$,
Theorem \ref{thm:RotatingOseen_WholeSpace_resolvent} implies 
\[
\rotdertermFTsimple{\uvel},\,\grad^2\uvel,\,\partial_1\uvel,\,\grad\upres
\in\LR{r}(\R^3)
\]
if additionally $r<2$.
Due to $\vvel=\uvel$ and $\vpres=\upres$ on $\ball^{2R}$, we have
\begin{align}\label{el:RotatingOseen_Regulartiy_OuterFunction}
\rotdertermFTsimple{\vvel},\,\grad^2\vvel,\,\partial_1\vvel,\,\grad\vpres
\in\LR{r}(\ball^{2R})
\end{align}
for $1< r\leq\frac{3}{2}q$ with $r<2$.
 
We combine \eqref{el:RotatingOseen_Regularity_InnerFunction} and
\eqref{el:RotatingOseen_Regulartiy_OuterFunction} to deduce
\begin{align}\label{el:RotatingOseen_Regulartiy_combined}
\rotdertermFTsimple{\vvel},\,\grad^2\vvel,\,\partial_1\vvel,\,\grad\vpres
\in\LR{r}(\Omega)
\end{align}
for $1<r\leq\frac{3}{2}q$ with $r<2$.
After repeating the above argument a sufficient number of times, we obtain
\eqref{el:RotatingOseen_Regulartiy_combined} for all $r\in(1,2)$.
Since $\vvel\in\LR{s}(\Omega)$, the Sobolev inequality further yields 
\[
\forall r\in\bp{\frac{3}{2},6}: \ \grad\vvel\in\LR{r}(\Omega), \qquad
\forall r\in(3,\infty) : \ \vvel\in\LR{r}(\Omega).
\]
In particular, we can employ the divergence theorem to compute
\[
\int_{\Omega_R} \Div \bb{\np{\eone\wedge x}\snorm{\vvel}^2} \,\dx
=\int_{\partial\Omega_R} \np{\eone\wedge x}\cdot \nvec \snorm{\vvel}^2 \,\dS
=\int_{\partial\ball_R} \np{\eone\wedge x}\cdot x R^{-1}\snorm{\vvel}^2 \,\dS
=0
\]
for any $R>0$ with $\partial\ball_R\subset\Omega$.
Passing to the limit $R\to\infty$, we obtain
\begin{align}\label{eq:RotatingOseen_Uniqueness_TestingE1wedgex}
\int_{\Omega} \Div \bb{\np{\eone\wedge x}\snorm{\vvel}^2} \,\dx=0.
\end{align}
By the above integrability properties, we can further multiply \eqrefsub{sys:RotatingOseen_resolvent}{1}
by $\vvel$ and integrate over $\Omega$. 
Utilizing \eqref{eq:RotatingOseen_Uniqueness_TestingE1wedgex} and integration by parts, we conclude
\begin{align*}
0 
&= \int_{\Omega}
\bp{
\tay \rotdertermFT{\vvel}
-\Delta\vvel
+\rey\partial_1\vvel
+\grad\vpres}
\cdot\vvel
\,\dx
\\
&= \int_{\Omega} 
\tay ik \,\snorm{\vvel}^2 
+\half\tay \Div \bb{\np{\eone\wedge x}\snorm{\vvel}^2}
-\Delta\vvel\cdot\vvel
+\half \rey\partial_1 \snorm{\vvel}^2
+\grad\vpres\cdot\vvel
\,\dx
\\
&= \tay ik\int_{\Omega}\snorm{\vvel}^2\,\dx
+\int_{\Omega}
\snorm{\grad\vvel}^2 \,\dx. 
\end{align*}
This implies $\grad\vvel=0$. 
The imposed boundary conditions thus yield $\vvel=0$.
Finally, \eqrefsub{sys:RotatingOseen_resolvent}{1} leads to $\grad\vpres=0$,
and the proof is complete.
\end{proof}

\subsection{A priori estimate}

Next we establish an \textit{a priori} estimate for the solution to 
the resolvent problem \eqref{sys:RotatingOseen_resolvent}.

\begin{lem}\label{lem:RotatingOseen_ResolventEst_ErrorTerms}
Let $q\in(1,2)$, $k\in\Z$ and $\rey,\,\tay, \,\theta>0$ with $\rey^2\leq\theta\tay$.
Moreover, let $\vf\in\LR{q}(\Omega)$ and $R>0$ such that $\partial\ball_R\subset\Omega$. 
Let $\np{\vvel,\vpres}\in\LRloc{1}(\Omega)$ with
\begin{align}\label{el:RotatingOseen_resolventProblem_fctClass}
\rotdertermFTsimple\vvel,\,
\grad^2\vvel,\,
\partial_1\vvel,\,
\grad\vpres\in\LR{q}(\Omega),\quad
\vvel\in\LR{s_1}(\Omega),\quad
\grad\vvel\in\LR{s_2}(\Omega)
\end{align}
be a solution to \eqref{sys:RotatingOseen_resolvent}.
Then there exists a constant
$\Cc[const:RotatingOseen_linearEst_ErrorTerms]{C}
=\const{const:RotatingOseen_linearEst_ErrorTerms}(\Omega,q,\theta,R)>0$
such that
\begin{align}\label{est:RotatingOseen_ResolventEst_ErrorTerms}
\begin{split}
\tay\norm{\rotdertermFTsimple\vvel}_{q}
&+\norm{\grad^2\vvel}_{q}
+\rey\norm{\partial_1\vvel}_{q}
+\rey^{1/2}\norm{\vvel}_{s_1}
+\rey^{1/4}\norm{\grad\vvel}_{s_2} 
+\norm{\grad\vpres}_{q}\\
\leq\const{const:RotatingOseen_linearEst_ErrorTerms}
\bp{\norm{\vf}_{q}&+\np{1+\rey+\tay}\norm{\vvel}_{1,q;\Omega_{4R}}
+\tay\snorm{k}\,\norm{\vvel}_{-1,q;\Omega_{4R}}
+\norm{\vpres}_{q;\Omega_{4R}}}.
\end{split}
\end{align}
\end{lem}

\begin{proof}
Let $\cutoff_0$, $\cutoff_1$ be the ``cut-off'' functions from the proof of Lemma \ref{lem:RotatingOseen_Uniqueness_resolvent}.
Define $\wvel\in\WSR{2}{q}(\Omega)$ and $\wpres\in\WSR{1}{q}(\Omega)$ 
as in \eqref{eq:RotatingOseen_Cutoff_InnerFunction}. Then
\[
\begin{pdeq}
i k\tay\,\wvel- \Delta \wvel + \grad \wpres &= \wf
&& \tin \Omega_{4R}, \\
\Div\wvel&=0
&& \tin \Omega_{4R}, \\
\wvel&=0
&& \ton \partial\Omega_{4R},
\end{pdeq}
\]
with 
\[
\wf\coloneqq
\bp{\vf-\tay\rotterm{\vvel}-\rey\partial_1\vvel}\cutoff_0
-2\grad\cutoff_0\cdot\grad\vvel-\Delta\cutoff_0\vvel+\grad\cutoff_0\vpres
-\np{ik\tay-\Delta}\bogopr(\vvel\cdot\grad\cutoff_0).
\]
Well-known theory for the Stokes resolvent problem 
(see for example \cite{FarwigSohr_GenResEstStokes1994}) 
yields
\begin{align}\label{est:RotatingOseen_Cutoff_InnerFunction}
\begin{split}
\norm{\vvel}_{2,q;\Omega_{2R}}
&+\norm{\grad\vpres}_{q;\Omega_{2R}}
\leq
\norm{\wvel}_{2,q;\Omega_{4R}}
+\norm{\grad\wpres}_{q;\Omega_{4R}}
\leq\Cc{c}\norm{h}_{q;\Omega_{4R}}
\\
&\quad\leq\Cc{c}\bp{
\norm{\vf}_q
+\np{1+\rey+\tay}
\norm{\vvel}_{1,q;\Omega_{4R}}
+\norm{\vpres}_{q;\Omega_{4R}}
+\tay\snorm{k}\,\snorm{\vvel\cdot\grad\chi_0}_{-1,q;\Omega_{4R}}^\ast
}.
\end{split}
\end{align}
In the last estimate we used mapping properties of the Bogovski\u\i{} operator 
(see \cite[Section III.3]{GaldiBookNew}), namely
\[
\norm{\grad\bogopr h}_{m,q;\Omega_{4R}}\leq\Cc{c}\norm{h}_{m,q;\Omega_{4R}},
\qquad
\norm{\bogopr h}_{q;\Omega_{4R}}\leq\Cc{c}\snorm{h}_{-1,q;\Omega_{4R}}^\ast
\]
for $m\in\N_0$, where 
\[
\snorm{h}_{-1,q;D}^\ast
\coloneqq
\sup\setcL{\snorml{\int_{D}h\psi\,\dx}}{\psi\in\CRci(\R^3),\ \norm{\grad\psi}_{q;D}=1}.
\]
To estimate the last term in \eqref{est:RotatingOseen_Cutoff_InnerFunction}, 
we introduce the notation 
\[
\overline{\psi}\coloneqq\psi-\frac{1}{\snorml{\Omega_{4R}}}\int_{\Omega_{4R}}\psi\,\dx
\]
for $\psi\in\CRci(\R^3)$, 
and we employ that $\Div\vvel=0$ in $\Omega$ and $\vvel=0$ on $\partial\Omega$ to deduce the identity
\begin{align*}
\int_{\Omega_{4R}}\vvel\cdot\grad\cutoff_0 \psi \,\dx
&=\int_{\Omega_{4R}}\Div\np{\vvel\cutoff_0} \psi \,\dx
=-\int_{\Omega_{4R}}\cutoff_0 \vvel\cdot \grad\overline{\psi} \,\dx\\
&=\int_{\Omega_{4R}}\Div \np{\vvel\cutoff_0}\overline{\psi} \,\dx
=\int_{\Omega_{4R}}\vvel\cdot \grad\cutoff_0\overline{\psi} \,\dx.
\end{align*}
Since Poincar\'e's inequality yields
\[
\norm{\overline{\psi}{\grad\cutoff_0}}_{1,q';\Omega_{4R}}
\leq\Cc{c}\norm{\overline{\psi}}_{1,q';\Omega_{4R}}
\leq\Cc{c}\norm{\grad{\psi}}_{q';\Omega_{4R}},
\]
we have
\begin{align*}
\snorm{\vvel\cdot\grad\chi_0}_{-1,q;\Omega_{4R}}^\ast
&\leq \sup\setcl{\norm{\vvel}_{-1,q;\Omega_{4R}}\norm{\overline{\psi}{\grad\cutoff_0}}_{1,q';\Omega_{4R}}}
{\psi\in\CRci(\R^3),\ \norm{\grad\psi}_{q';\Omega_{4R}}=1}\\
&\leq\Cc{c} \norm{\vvel}_{-1,q;\Omega_{4R}}.
\end{align*}
Applying this estimate to the last term in \eqref{est:RotatingOseen_Cutoff_InnerFunction}, we obtain
\begin{align}\label{est:RotatingOseen_Cutoff_InnerFunction_BetterNorm}
\norm{\vvel}_{2,q;\Omega_{2R}}
+\norm{\grad\vpres}_{q;\Omega_{2R}}
\leq\Cc{c}\bp{
\norm{\vf}_q
+\np{1+\rey+\tay}
\norm{\vvel}_{1,q;\Omega_{4R}}
+\norm{\vpres}_{q;\Omega_{4R}}
+\tay\snorm{k}\,\norm{\vvel}_{-1,q;\Omega_{4R}}
}.
\end{align}

Next define ($\uvel,\upres)$ as in \eqref{eq:RotatingOseen_Cutoff_OuterFunction}, which
satisfies the system
\[
\begin{pdeq}
\tay\rotdertermFT{\uvel} -\Delta \uvel -\rey\partial_1\uvel + \grad \upres &= \uf
& \tin \R^3, \\
\Div\uvel&=0
& \tin \R^3,
\end{pdeq}
\]
with
\begin{align*}
\uf&\coloneqq
\cutoff_1\vf
-\tay\np{\eone\wedge x\cdot\grad\cutoff_1}\vvel
-2\grad\cutoff_1\cdot\grad\uvel-\Delta\cutoff_1\vvel
+\rey\partial_1\cutoff_1\vvel
+\grad\cutoff_1\vpres -\Delta\bogopr(\vvel\cdot\grad\cutoff_1) 
\\
&
\quad+\rey\partial_1\bogopr(\vvel\cdot\grad\cutoff_1)
+\tay\rotdertermFT{\bogopr(\vvel\cdot\grad\cutoff_1)}.
\end{align*}
Theorem \ref{thm:RotatingOseen_WholeSpace_resolvent} implies
\[
\begin{split}
\tay&\norm{\rotdertermFTsimple\vvel}_{q;\Omega^{2R}}
+\norm{\grad^2\vvel}_{q;\Omega^{2R}}
+\rey\norm{\partial_1\vvel}_{q;\Omega^{2R}}\\
&\qquad+\rey^{1/4}\norm{\grad\vvel}_{s_2;\Omega^{2R}} 
+\rey^{1/2}\norm{\vvel}_{s_1;\Omega^{2R}}
+\norm{\grad\vpres}_{q;\Omega^{2R}}
\\
&\leq\tay\norm{\rotdertermFTsimple\uvel}_{q}
+\norm{\grad^2\uvel}_{q}
+\rey\norm{\partial_1\uvel}_{q}
+\rey^{1/4}\norm{\grad\uvel}_{s_2} 
+\rey^{1/2}\norm{\uvel}_{s_1}
+\norm{\grad\upres}_{q}
\\
&\leq\Cc{c}\bp{
\norm{\vf}_q
+\np{1+\rey+\tay}\norm{\vvel}_{1,q;\Omega_{2R}}
+\norm{\vpres}_{q;\Omega_{2R}}
+\tay\snorm{k}\norm{\vvel}_{-1,q;\Omega_{2R}}
},
\end{split}
\]
where we estimated the terms containing the Bogovski\u\i{} operator as above.
Combining this estimate with \eqref{est:RotatingOseen_Cutoff_InnerFunction_BetterNorm},
we conclude \eqref{est:RotatingOseen_ResolventEst_ErrorTerms}.
\end{proof}

In the next step we improve estimate \eqref{est:RotatingOseen_ResolventEst_ErrorTerms}
by showing that the lower-order terms 
on the right-hand side can be omitted.
This leads to the desired estimate \eqref{est:RotatingOseen_resolvent}
with the asserted dependencies of the constant $\const{const:RotatingOseen_linear}$.

\begin{lem}\label{lem:RotatingOseen_ResolventEst_Final}
Let $q\in(1,2)$, $k\in\Z$ and $\rey,\,\tay>0$,
and let $\vf\in\LR{q}(\Omega)$. 
Let $\np{\vvel,\vpres}\in\LRloc{1}(\Omega)$ 
be a solution to \eqref{sys:RotatingOseen_resolvent} 
in the class \eqref{el:RotatingOseen_resolventProblem_fctClass}.
Then estimate \eqref{est:RotatingOseen_resolvent} holds
for a constant
$\const{const:RotatingOseen_linear}
=\const{const:RotatingOseen_linear}\np{\Omega,q,\rey,\tay}>0$.
If $q\in(1,\frac{3}{2})$
and
$\rey^2\leq\theta\tay\leq B$ then this constant
can be chosen independently of $\rey$ and $\tay$ such that
$\const{const:RotatingOseen_linear}
=\const{const:RotatingOseen_linear}(\Omega,q,\theta,B)$.
\end{lem}

\begin{proof}
We employ a contradiction argument. 
At first, consider the case $q\in(1,\frac{3}{2})$ and
assume that \eqref{est:RotatingOseen_resolvent} is not valid
for a constant 
$\const{const:RotatingOseen_linear}=\const{const:RotatingOseen_linear}\np{\Omega,q,\theta,B}$.
Then there exist sequences of numbers 
$(\rey_j)\subset(0,\sqrt{B}]$, 
$(\tay_j)\subset(0,B/\theta]$ 
with $\rey_j^2\leq\theta\tay_j$,
and 
$(k_j)\subset\Z$,
and of functions $(\vvel_j)$, $(\vpres_j)$, $(\vf_j)$
that satisfy
\begin{align}\label{eq:RotatingOseen_ResolventEst_Final:SeqNorm}
\begin{split}
\tay_j&\norm{\rotdertermFTsimpleseq{\vvel_j}}_{q}
+\norm{\grad^2\vvel_j}_{q}\\
&\qquad+\rey_j\norm{\partial_1\vvel_j}_{q} 
+\rey_j^{1/2}\norm{\vvel_j}_{s_1}
+\rey_j^{1/4}\norm{\grad\vvel_j}_{s_2} 
+\norm{\grad\vpres_j}_{q}
=1, 
\end{split}
\end{align}
$\norm{\vf_j}_{q}\to 0$ as $j\to\infty$, and
\begin{align}\label{sys:RotatingOseen_ResolventEst_Final:Seq}
\begin{pdeq}
\tay_j\rotdertermFTseq{\vvel_j}
- \Delta \vvel_j 
- \rey_j \partial_1 \vvel_j 
+ \grad \vpres_j 
&= \vf_j && \tin\Omega, \\
\Div\vvel_j&=0 && \tin\Omega,\\
\vvel_j&=0 && \ton \partial\Omega,
\end{pdeq}
\end{align}
for all $j\in\N$.
Furthermore, without loss of generality, 
we may assume $\int_{\Omega_R}\vpres_j\,\dx=0$ 
for $R>0$ as in Lemma \ref{lem:RotatingOseen_ResolventEst_ErrorTerms}.
Then, $(\rey_j)$, $(\tay_j)$ and $(k_j)$ contain (improper) convergent subsequences
with limits $\rey\in\nb{0,\sqrt{B}}$,
$\tay\in\nb{0,B/\theta}$
and $k\in\Z\cup\set{\pm\infty}$, 
respectively, and we have $\rey^2\leq\theta\tay$.
For simplicity, we identify selected subsequences with the actual sequences.
Moreover, \eqref{eq:RotatingOseen_ResolventEst_Final:SeqNorm} 
implies that 
$\Uvel_j\coloneqq(i\tay_j k_j\vvel_j,\vvel_j,\vpres_j)$ is bounded in 
$\LR{q}(\Omega_\rho)\times\WSR{2}{q}(\Omega_\rho)\times\WSR{1}{q}(\Omega_\rho)$
for any $\rho>R$.
Hence, by a Cantor diagonalization argument, there exists a subsequence 
that converges weakly in $\LR{q}(\Omega_\rho)\times\WSR{2}{q}(\Omega_\rho)\times\WSR{1}{q}(\Omega_\rho)$ to some $\Uvel\coloneqq(\wvel,\vvel,\vpres)$
for each $\rho>R$.
Consequently, passing to the limit $j\to\infty$ in
\eqref{sys:RotatingOseen_ResolventEst_Final:Seq},
we obtain
\begin{align}\label{sys:RotatingOseen_ResolventEst_Final:Limit}
\begin{pdeq}
\wvel
+\tay\rotterm{\vvel}
- \Delta \vvel 
- \rey \partial_1 \vvel 
+ \grad \upres 
&= 0 && \tin\Omega, \\
\Div\vvel&=0 && \tin\Omega,\\
\vvel&=0 && \ton \partial\Omega.
\end{pdeq}
\end{align}
Moreover, by the compact embeddings 
\[
\WSR{2}{q}(\Omega_{4R})\embeds
\WSR{1}{q}(\Omega_{4R})\embeds
\LR{q}(\Omega_{4R})\embeds
\WSR{-1}{q}(\Omega_{4R}),
\]
we deduce that $\Uvel$ is the strong limit of 
$(\Uvel_j)$ 
in the topology of $\WSR{-1}{q}(\Omega_{4R})\times\WSR{1}{q}(\Omega_{4R})\times\LR{q}(\Omega_{4R})$.
By Lemma \ref{lem:RotatingOseen_ResolventEst_ErrorTerms}, 
\begin{align*}
\begin{split}
&\tay_j\norm{ \rotdertermFTsimpleseq{\vvel_j}}_{q}
+\norm{\grad^2\vvel_j}_{q}\\
&\qquad +\rey_j\norm{\partial_1\vvel_j}_{q}
+\rey_j^{1/2}\norm{\vvel_j}_{s_1}
+\rey_j^{1/4}\norm{\grad\vvel_j}_{s_2} 
+\norm{\grad\vpres_j}_{q}\\
&\qquad \qquad\leq\const{const:RotatingOseen_linearEst_ErrorTerms}
\bp{\norm{\vf_j}_{q}+\np{1+\rey_j+\tay_j}\norm{\vvel_j}_{1,q;\Omega_{4R}}
+\tay\snorm{k_j}\,\norm{\vvel_j}_{-1,q;\Omega_{4R}}
+\norm{\vpres_j}_{q;\Omega_{4R}}}.
\end{split}
\end{align*}
Passing to the limit $j\to\infty$ in this estimate,
we conclude in virtue of \eqref{eq:RotatingOseen_ResolventEst_Final:SeqNorm} that
\begin{align}\label{eq:RotatingOseen_ResolventEst_Final:Limit}
1
\leq\const{const:RotatingOseen_linearEst_ErrorTerms}
\bp{\np{1+\rey+\tay}\norm{\vvel}_{1,q;\Omega_{4R}}
+\norm{\wvel}_{-1,q;\Omega_{4R}}
+\norm{\vpres}_{q;\Omega_{4R}}}.
\end{align}
Moreover,
\begin{align}\label{Mads1}
\begin{split}
\norm{\wvel + \tay\rotterm\vvel}_{q}
&+\norm{\grad^2\vvel}_{q}
+\rey\norm{\partial_1\vvel}_{q}
+\rey^{1/2}\norm{\vvel}_{s_1}
+\rey^{1/4}\norm{\grad\vvel}_{s_2} 
+\norm{\grad\vpres}_{q}< \infty.
\end{split}
\end{align}
Now we distinguish between several cases:
\begin{enumerate}[i.]
\item
If $\tay_j k_j\to s\in\R$ and $\tay=0$, then $\rey=0$ and $\wvel=is\vvel$, so that
\eqref{sys:RotatingOseen_ResolventEst_Final:Limit}
reduces to a Stokes resolvent problem.
If $s\neq 0$, we also have $\vvel\in\LR{q}(\Omega)$ 
and we conclude $\vvel=\grad\vpres=0$ from a well-known uniqueness result;
see for example \cite{FarwigSohr_GenResEstStokes1994}.
If $s=0$, we utilize that $q<\frac32$ and $\vvel_j\in\LR{s_1}(\Omega)$, $\grad\vvel_j\in\LR{s_2}(\Omega)$, 
so that Sobolev's inequality implies 
\[
\norm{\vvel_j}_{3q/(3-2q)}\leq\Cc{c}\norm{\grad\vvel_j}_{3q/(3-q)}\leq\Cc{c}\norm{\grad^2\vvel_j}_{q},
\]
and thus $\vvel\in\LR{3q/(3-2q)}(\Omega)$.
Now $\vvel=\grad\vpres=0$ follows from classical uniqueness properties of the steady-state Stokes problem, 
see for example \cite[Theorem V.4.6]{GaldiBookNew}.
\item
If $\tay_j k_j\to s\in\R$ and $\tay\neq 0$ but $\rey=0$, 
then $k_j\to k\in\Z$ and $\wvel=i\tay k\vvel$, so that
\eqref{sys:RotatingOseen_ResolventEst_Final:Limit}
reduces to \eqref{sys:RotatingOseen_resolvent} with $\rey=0$.
As above, we deduce $\vvel\in\LR{3q/(3-2q)}(\Omega)$.  
From Lemma \ref{lem:RotatingOseen_Uniqueness_resolvent} we conclude
$\vvel=\grad\vpres=0$.
\item
If $\tay_j k_j\to s\in\R$ and $\tay\neq 0$ and $\rey\neq 0$, 
then $k_j\to k\in\Z$ and $\wvel=i\tay k\vvel$, 
so that
$\np{\vvel,\vpres}$
satisfies \eqref{sys:RotatingOseen_resolvent}.
Since $\rey\neq 0$, it follows from \eqref{Mads1} that $\vvel\in\LR{s_1}(\Omega)$.
Lemma \ref{lem:RotatingOseen_Uniqueness_resolvent} thus
implies $\vvel=\grad\vpres=0$.
\item
If $\tay_j \snorm{k_j}\to\infty$, 
we recall \eqref{eq:RotatingOseen_ResolventEst_Final:SeqNorm} and estimate
\[
\tay_j\snorm{k_j}\norm{\vvel_j}_{q;\Omega_\rho}
\leq
\tay_j\norm{\rotdertermFTsimpleseq{\vvel_j}}_{q;\Omega_\rho}
+\Cc{c}(\rho)\norm{\vvel_j}_{1,q;\Omega_\rho}
\leq \Cc{c}(\rho)
\]
for any $\rho>R$.
Passing to the limit $j\to\infty$, we thus obtain $\vvel=0$ on $\Omega_\rho$
for each $\rho>R$, whence $\vvel= 0$ on $\Omega$.
Hence, \eqrefsub{sys:RotatingOseen_ResolventEst_Final:Limit}{1} reduces to $\wvel+\grad\vpres=0$.
Clearly, we also have $\Div\wvel=0$ and $\restriction{\wvel}{\partial\Omega}=0$,
so that $\wvel+\grad\vpres=0$ corresponds to the Helmholtz decomposition of $0$ in $\LR{q}(\Omega)$.
Since this decomposition is unique, we conclude $\wvel=\grad\vpres=0$.
\end{enumerate}
Consequently, all four cases lead to $\wvel=\vvel=\grad\vpres=0$,
which contradicts  
\eqref{eq:RotatingOseen_ResolventEst_Final:Limit}.
This completes the proof in the case $1<q<\frac{3}{2}$.

In the more general case $q\in(1,2)$,
where we do not assert the constant $\const{const:RotatingOseen_linear}$ to be independent of $\rey$ and $\tay$,
these parameters remain fixed in the contradiction argument above.
Consequently, only the last two cases above have to be considered. The conclusion in both of these cases is valid for all $q\in(1,2)$, and we thus conclude the lemma. 
\end{proof}

\subsection{Existence}
To complete the proof of Theorem \ref{thm:RotatingOseen_resolvent}, 
it remains to show existence of a solution. For this purpose,
recall the following property of the Stokes operator. 

\begin{lem}\label{lem:EstSecondDerivativesStokesOp}
Let $D\subset\R^3$ be a bounded domain with $\CR{3}$-boundary. 
Every $\uvel\in\LRsigma{2}(D)\cap\WSRN{1}{2}(D)\cap\WSR{2}{2}(D)$ 
satisfies
\[
\norm{\grad^2\uvel}_2
\leq \Cc[const:EstSecondDerivativesStokesOp]{C}
\bp{\norm{\projh\Delta\uvel}_2+\norm{\grad\uvel}_2}
\]
for a constant 
$\const{const:EstSecondDerivativesStokesOp}=\const{const:EstSecondDerivativesStokesOp}(D)>0$ 
that does not depend on the ``size'' of $D$ 
but solely on its ``regularity''.
In particular, if $D=\Omega_R$ for an exterior domain $\Omega$ with $\partial\Omega\subset\ball_R$, 
the constant $\const{const:EstSecondDerivativesStokesOp}$ is independent of $R$
and solely depends on $\Omega$.
\end{lem}
\begin{proof}
See \cite[Lemma 1]{Heywood1980}.
\end{proof}

We further need the following identity from \cite{GaldiSilvestre_StrongSolNSObstacle_2005}.

\begin{lem}\label{lem:PropertiesRotTerm}
Let $\uvel\in\LRsigma{2}(\Omega_R)\cap\WSRN{1}{2}(\Omega_R)\cap\WSR{2}{2}(\Omega_R)$
with complex conjugate $\uvel^\ast$. 
Then $\rottermsimple{\uvel}\in\LRsigma{2}(\Omega_R)$ and
\begin{align*}
\int_{\Omega_R}&\rotterm{\uvel}\cdot
\projh\Delta\uvel^\ast\,\dx
\\
&=\int_{\partial\Omega} 
\frac{1}{2}\snorm{\grad\uvel}^2 \np{\eone\wedge x}\cdot \nvec
-\nvec \cdot\grad\uvel^\ast\cdot
\np{\eone\wedge x \cdot\grad\uvel}\,\dS
- \int_{\Omega_R} \grad\np{\eone\wedge\uvel}:\grad\uvel^\ast\,\dx.
\end{align*}
\end{lem}

\begin{proof}
See \cite[Lemma 3]{GaldiSilvestre_StrongSolNSObstacle_2005}.
\end{proof}

Existence of a solution to the resolvent problem
\eqref{sys:RotatingOseen_resolvent}
can be shown via a Galerkin approach combined with
an ``invading domains'' technique.

\begin{lem}\label{lem:RotatingOseen_Existence}
Let $\Omega\subset\R^3$ be an exterior domain of class $\CR{3}$.
Let $\rey,\,\tay>0$, $k\in\Z$,
and let $\vf\in\CRci(\Omega)$. 
Then there exists a solution 
$\np{\vvel,\vpres}$ to \eqref{sys:RotatingOseen_resolvent}
with 
\[
\rotdertermFTsimple\vvel,\,
\grad^2\vvel,\,
\partial_1\vvel,\,
\grad\vpres\in\LR{q}(\Omega),\quad
\vvel\in\LR{2q/(2-q)}(\Omega),\quad
\grad\vvel\in\LR{4q/(4-q)}(\Omega)
\]
for all $q\in(1,2)$.
\end{lem}

\begin{proof}
Let $R>0$ such that $\partial\ball_R\subset\Omega$, and take 
$m\in\N$ with $m>2R$. 
Since the Stokes operator in the bounded domain $\Omega_m$ 
is a positive self-adjoint invertible operator (see \cite[Chapter III, Theorem 2.1.1]{SohrBook}),
there exists a sequence $(\psi_j)_{j\in\N}$ of (real valued) eigenfunctions 
and $(\mu_j)_{j\in\N}\subset(0,\infty)$ of eigenvalues, that is,
\[
-\projh \Delta \psi_j=\mu_j\psi_j, \qquad 
\psi_j\in\LRsigma{2}(\Omega_m)\cap\WSRN{1}{2}(\Omega_m)\cap\WSR{2}{2}(\Omega_m), 
\]
normalized such that
\[
\int_{\Omega_m} \psi_j \cdot \psi_\ell \,\dx= \frac{1}{\mu_j}\delta_{j\ell}.
\]
We show the existence of a function 
$\uvel=\uvel^m_n\in\spanspacemn
\coloneqq
\vecspan_{\C}\setcl{\psi_j}{j=1,\ldots,n}$
satisfying
\begin{align}\label{eq:RotatingOseen_Resolvent:ApproxSol}
\int_{\Omega_m}
\bb{
\tay\rotdertermFT{\uvel}
-\Delta\uvel
-\rey\partial_1\uvel
}
\cdot\psi_j
\,\dx
=\int_{\Omega_m} \vf\cdot\psi_j\,\dx
\end{align}
for all $j\in \set{1,\ldots,n}$.
Since
\[
\uvel=\sum_{\ell=1}^n\xi_\ell\psi_{\ell}
\]
for some $\xi_1,\ldots,\xi_n\in\C$, this is equivalent to solving the algebraic equation
\begin{align}\label{eq:RotatingOseen_Resolvent:GalerkinAlgebraic}
(\idmatrix + M) \xi= c
\end{align}
with $\xi=(\xi_1,\ldots,\xi_n) \in\C^n$ and
\begin{align*}
M&=(M_{\ell j})\in\C^{n\times n}, &
M_{\ell j} 
&\coloneqq
\int_{\Omega_m} \bp{
\tay\rotdertermFT{\psi_\ell}
-\rey \partial_1 \psi_\ell
}\cdot\psi_j\,\dx, \\
c&=(c_j)\in\C^n, &
c_j
&\coloneqq
\int_{\Omega_m} \vf\cdot\psi_j\,\dx.
\end{align*}
Note that \eqref{eq:RotatingOseen_Resolvent:GalerkinAlgebraic} is a resolvent problem for the skew-Hermitian matrix $M$, 
which is uniquely solvable.
Existence of a unique solution $\uvel=\uvel^m_n\in\spanspacemn$ to
\eqref{eq:RotatingOseen_Resolvent:ApproxSol} thus follows.

Next we need suitable estimates for $\uvel=\uvel^m_n$.
Multiplication of both sides of \eqref{eq:RotatingOseen_Resolvent:ApproxSol} 
by the complex conjugate coefficient $\xi_j^\ast$ 
and summation over $j=1,\dots,n$ yields
\[
\norm{\grad\uvel}_2^2
+ \int_{\Omega_m}\bp{
\tay\rotdertermFT{\uvel}
-\rey\partial_1\uvel
}
\cdot\uvel^\ast\,\dx
= \int_{\Omega_m} \vf\cdot\uvel^\ast \,\dx.
\]
Because the integral term on the left-hand side is purely imaginary,
taking the real part of this equation 
leads to the estimate
\[
\norm{\grad\uvel}_2^2
\leq \norm{\vf}_{6/5}\norm{\uvel}_{6}.
\] 
Recalling the Sobolev inequality
$\norm{\uvel}_{6}\leq\Cc{c}\norm{\grad\uvel}_2$, 
we obtain
\begin{align}\label{est:RotatingOseen_Resolvent:FirstDer}
\norm{\uvel}_{6}
+\norm{\grad\uvel}_2
\leq \Cc[const:RotatingOseen_linear:FirstDer]{c}
\norm{\vf}_{6/5},
\end{align}
where $\const{const:RotatingOseen_linear:FirstDer}$ is independent of $m$.
If we  multiply both sides of \eqref{eq:RotatingOseen_Resolvent:ApproxSol} 
by $\mu_j\xi_j^\ast$
and sum over $j=1,\dots,n$, we obtain
\[
\norm{\projh\Delta\uvel}_2^2
= 
\int_{\Omega_m} \bb{\vf
-\tay\rotdertermFT{\uvel}
+\rey\partial_1\uvel}
\cdot\projh\Delta\uvel^\ast\,\dx.
\]
Taking real part of both sides and observing that
\[
\realpart \int_{\Omega_m} ik\uvel\cdot\projh\Delta\uvel^\ast\,\dx
= - \realpart \bp{ik \norm{\grad\uvel}_{2}^2} =0,
\]
we conclude using H\"older's inequality the estimate
\begin{align}\label{est:RotatingOseen_Resolvent:StokesOp}
\norm{\projh\Delta\uvel}_2^2
\leq \bp{
\norm{\vf}_2 +\rey\norm{\partial_1\uvel}_2}
\norm{\projh\Delta\uvel}_2 +\realpart
\int_{\Omega_m} \tay\rotterm{\uvel}
\cdot\projh\Delta\uvel^\ast\,\dx.
\end{align}
Using Lemma \ref{lem:PropertiesRotTerm}, we estimate the remaining integral on the right-hand side to conclude
\[
\realpart\int_{\Omega_m}\tay\rotterm{\uvel}\cdot
\projh\Delta\uvel^\ast\,\dx
\leq \Cc{c}\tay\bp{
\norm{\grad\uvel}_{2;\partial\Omega}^2
+\norm{\grad\uvel}_{2;\Omega_m}^2
}
\]
with $\Cclast{c}$ independent of $m$.
Employing the trace inequality \cite[Theorem II.4.1]{GaldiBookNew} on the domain $\Omega_{R}$, we further estimate
\begin{align*}
\realpart\int_{\Omega_m}\tay\rotterm{\uvel}\cdot\projh\Delta\uvel^\ast\,\dx
&\leq \Cc{c}\tay\bp{
\norm{\grad\uvel}_{2;\Omega_{R}}
\norm{\grad\uvel}_{1,2;\Omega_{R}}
+\norm{\grad\uvel}_{2;\Omega_m}^2
}
\\
&\leq \Cc{c}(\varepsilon)
\np{\tay+\tay^2}\norm{\grad\uvel}_{2;\Omega_m}^2
+\varepsilon\norm{\grad^2\uvel}_{2;\Omega_m}^2
\end{align*} 
for small $\varepsilon>0$. 
From Lemma \ref{lem:EstSecondDerivativesStokesOp} we deduce
\[
\realpart\int_{\Omega_m}\tay\rotterm{\uvel}\cdot\projh\Delta\uvel^\ast\,\dx
\leq \Cc{c}(\varepsilon)\np{\tay+\tay^2}
\norm{\grad\uvel}_{2;\Omega_m}^2
+\varepsilon\Cc[const:RotatingOseen_linear:EstimateStokesOp]{c}\norm{\projh\Delta\uvel}_{2;\Omega_m}^2
\]
with a constant $\const{const:RotatingOseen_linear:EstimateStokesOp} >0$ independent of $m$.
Combining this estimate with \eqref{est:RotatingOseen_Resolvent:StokesOp}, 
choosing $\varepsilon$ sufficiently small
and employing estimate \eqref{est:RotatingOseen_Resolvent:FirstDer},
we arrive at
\[
\norm{\projh\Delta\uvel}_{2;\Omega_m} 
\leq
\Cc{c}
\bp{1+\rey+\sqrt{\tay+\tay^2}}
\bp{\norm{\vf}_2 + \norm{\vf}_{6/5}}.
\]
Using Lemma \ref{lem:EstSecondDerivativesStokesOp} 
and estimate \eqref{est:RotatingOseen_Resolvent:FirstDer} once again
and restoring the original notation,
we end up with
\begin{align}\label{est:RotatingOseen_Resolvent:SecondDerivatives}
\norm{\grad^2\uvel^m_n}_{2;\Omega_m}
\leq 
\Cc{c}\bp{ 
\norm{\projh\Delta\uvel^m_n}_{2;\Omega_m} 
+\norm{\grad\uvel^m_n}_{2;\Omega_m}
}
\leq
\Cc{c}
\bp{\norm{\vf}_2 + \norm{\vf}_{6/5}}
\end{align}
with $\Cclast{c}$ independent of $m$.

In particular, we see from \eqref{est:RotatingOseen_Resolvent:FirstDer},
\eqref{est:RotatingOseen_Resolvent:SecondDerivatives}
and Poincar\'e's inequality
that $(\uvel^m_n)$ is uniformly bounded in $\WSR{2}{2}(\Omega_m)$ 
and thus contains a subsequence that converges weakly to some function
$\vvel^m\in\LRsigma{2}(\Omega_m)\cap\WSRN{1}{2}(\Omega_m)\cap\WSR{2}{2}(\Omega_m)$,
which obeys the estimate
\begin{align}\label{est:RotatingOseen_Resolvent:WeakLimitOmegam}
\norm{\vvel^m}_{6;\Omega_m}
+\norm{\grad\vvel^m}_{1,2;\Omega_m}
\leq\Cc[const:RotatingOseen_linear:WeakLimitOmegam]{c}\bp{\norm{\vf}_{6/5}+\norm{\vf}_{2}}
\end{align}
with $\const{const:RotatingOseen_linear:WeakLimitOmegam}$ independent of $m$.
Moreover, $\vvel^m$ 
satisfies \eqref{eq:RotatingOseen_Resolvent:ApproxSol} for all $j\in\N$, 
whence there exists $\vpres^m\in\WSR{1}{2}(\Omega_m)$ 
such that
\begin{align}\label{sys:RotatingOseen_Resolvent:WeakLimitOmegam}
\begin{pdeq}
\tay\rotdertermFT{\vvel^m}
- \Delta \vvel^m 
- \rey \partial_1 \vvel^m 
+ \grad \vpres^m 
&=\vf && \tin\Omega_m, \\
\Div\vvel^m&=0 && \tin\Omega_m,\\
\vvel^m&=0 && \ton\partial\Omega_m;
\end{pdeq}
\end{align}
see \cite[Corollary III.5.1]{GaldiBookNew}.
Since $\rottermsimple{\vvel^m}\in\LRsigma{2}(\Omega_m)$ by Lemma \ref{lem:PropertiesRotTerm},
we deduce from \eqref{sys:RotatingOseen_Resolvent:WeakLimitOmegam} and \eqref{est:RotatingOseen_Resolvent:WeakLimitOmegam} the estimate
\begin{align*}
\tay&\norm{\rotdertermFTsimple{\vvel^m}}_{2}
=\tay\norml{\projh \rotdertermFT{\vvel^m}}_{2} \\
&\leq \norm{\projh \vf}_{2}
+\norm{\projh\Delta\vvel^m}_{2}
+ \rey\norm{\projh \partial_1\vvel^m}_{2}
\leq\Cc{c}\bp{\norm{\vf}_{6/5}+\norm{\vf}_{2}}.
\end{align*}
Combining the estimate above with \eqref{est:RotatingOseen_Resolvent:WeakLimitOmegam}, we conclude
\begin{align}\label{est:RotatingOseen_Resolvent:WeakLimitOmegam_Complete}
\begin{split}
\norm{\vvel^m}_{6;\Omega_m}
+\norm{\grad\vvel^m}_{1,2;\Omega_m}
&+\tay\norm{\rotdertermFTsimple{\vvel^m}}_{2;\Omega_m}\\
&\qquad\qquad\qquad\qquad\quad
\leq\Cc[const:RotatingOseen_linear:WeakLimitOmegamWithRot]{c}\bp{\norm{\vf}_{6/5}+\norm{\vf}_{2}}
\end{split}
\end{align}
with $\const{const:RotatingOseen_linear:WeakLimitOmegamWithRot}$ independent of $m$.

Now we introduce a sequence of
rotationally symmetric ``cut-off'' functions $(\cutoff_m)\subset\CRci(\R^3)$
satisfying 
\[
\cutoff_m(x)=1 \ \tfor \snorm{x}\leq \frac{m}{2}, \qquad  
\cutoff_m(x)=0 \ \tfor \snorm{x}\geq \frac{3m}{4}, \qquad
\snorm{\grad\cutoff_m}\leq \frac{\Cc{c}}{m}, \qquad
\snorm{\grad^2\cutoff_m}\leq \frac{\Cc{c}}{m^2},
\]
and we set $\wvel^m\coloneqq\cutoff_m\vvel^m$.
Then $\wvel^m$ is an element of $\WSR{2}{2}(\Omega)$.
Moreover, the rotational symmetry of $\cutoff_m$ implies $\eone\wedge x\cdot\grad\cutoff_m=0$.
Therefore, from \eqref{est:RotatingOseen_Resolvent:WeakLimitOmegam_Complete} 
and the properties of $\cutoff_m$, 
we deduce the estimate
\begin{align*}
\norm{\wvel^m}_{6}
+\norm{\grad\wvel^m}_{1,2}
+\tay\norm{\rotdertermFTsimple{\wvel^m}}_{2}
&\leq\Cc{c}
\bp{\norm{\vf}_{6/5}+\norm{\vf}_{2}}
\end{align*}
with $\Cclast{c}$ independent of $m$.
This implies the existence of a subsequence, still denoted by $(\wvel^m)$, 
that converges in the sense of distributions 
to some function $\vvel\in\WSRloc{2}{2}(\Omega)$
that satisfies 
\begin{align}\label{est:RotatingOseen_Resolvent:FinalWeakLimit}
\norm{\vvel}_{6}
+\norm{\grad\vvel}_{1,2}
+\tay\norm{\rotdertermFTsimple{\vvel}}_{2}
&\leq\const{const:RotatingOseen_linear:WeakLimitOmegamWithRot}
\bp{\norm{\vf}_{6/5}+\norm{\vf}_{2}}.
\end{align}
Moreover, $\restriction{\vvel}{\partial\Omega}=0$.
Let $\varphi\in\CRci(\Omega)$. 
We choose $m_0\in\N$ such that $\supp\varphi$ is contained in $\Omega_{m_0/2}$.
For $m\geq m_0$ we have $\wvel^m=\vvel^m$ on $\Omega_{m_0/2}$ and thus
\[
\int_{\Omega} \wvel^m \cdot\grad\varphi\,\dx
=\int_{\Omega} \vvel^m \cdot\grad\varphi\,\dx
= 0
\]
by \eqrefsub{sys:RotatingOseen_Resolvent:WeakLimitOmegam}{2}.
Passing to the limit $m\to\infty$, we conclude $\Div\vvel=0$.
Now let $\psi\in\CRcisigma(\Omega)$ 
and choose $m_0$ such that $\supp\psi\subset\Omega_{m_0/2}$.
With the same argument as above, 
for $m\geq m_0$ we obtain from \eqrefsub{sys:RotatingOseen_Resolvent:WeakLimitOmegam}{1}
that
\begin{align*}
&\int_{\Omega}
\bp{
\tay\rotdertermFT{\wvel^m} - \Delta \wvel^m - \rey \partial_1 \wvel^m -\vf
}\cdot\psi\,\dx
\\
&\quad=
\int_{\Omega}
\bp{
\tay\rotdertermFT{\vvel^m} - \Delta \vvel^m - \rey \partial_1 \vvel^m +\grad\vpres^m -\vf
}\cdot\psi\,\dx=0.
\end{align*}
Therefore, by passing to the limit $m\to\infty$, 
we see
\[
\int_{\Omega}
\bp{
\tay\rotdertermFT{\vvel} - \Delta \vvel - \rey \partial_1 \vvel -\vf
}\cdot\psi\,\dx=0
\]
for all $\psi\in\CRcisigma(\Omega)$.
Consequently, by Helmholtz decomposition, 
there exists a function $\vpres$ 
with $\grad\vpres\in\LR{2}(\Omega)$ 
such that $(\vvel,\vpres)$ is a solution to \eqref{sys:RotatingOseen_resolvent}.

It remains to show that $\vvel$ and $\vpres$
belong to the correct function spaces.
By H\"older's inequality, we directly find that 
\begin{align}\label{Mads2}
\vvel\in\WSR{2}{q}(\Omega_{\rho}), 
\qquad 
\vpres\in\WSR{1}{q}(\Omega_{\rho})
\end{align}
for any $\rho>R$ and all $q\in[1,2]$.
Repeating the ``cut-off'' argument from \eqref{eq:RotatingOseen_Cutoff_OuterFunction},
we obtain $\np{\uvel,\upres}$ which satisfy \eqref{sys:RotatingOseen_OuterFunction}
for some function $f\in\LR{2}(\R^3)$ with compact support.
In particular, this implies $f\in\LR{q}(\R^3)$ for all $q\in(1,2)$.
Theorem \ref{thm:RotatingOseen_WholeSpace_resolvent} yields existence of a solution to \eqref{sys:RotatingOseen_OuterFunction}
satisfying \eqref{est:RotatingOseen_WholeSpace_resolvent}. 
Since $\uvel\in\LR{6}(\R^3)$, Theorem \ref{thm:RotatingOseen_WholeSpace_resolvent} further ensures that 
$(\uvel,\upres)$ coincides with this solution.
We thus have
\[
\rotdertermFTsimple\uvel,\,
\grad^2\uvel,\,
\partial_1\uvel,\,
\grad\upres\in\LR{q}(\R^3),\
\uvel\in\LR{2q/(2-q)}(\R^3),\
\grad\uvel\in\LR{4q/(4-q)}(\R^3)
\] 
Since $\vvel=\uvel$ and $\vpres=\upres$ on $\ball^{2R}$, the integrability properties above in combination with \eqref{Mads2} show that
$\vvel$ and $\vpres$ belong to the correct function spaces. 
\end{proof}

Combining Lemma \ref{lem:RotatingOseen_Uniqueness_resolvent}, 
Lemma \ref{lem:RotatingOseen_ResolventEst_Final} 
and Lemma \ref{lem:RotatingOseen_Existence},
we can finally complete the proof of Theorem \ref{thm:RotatingOseen_resolvent}.

\begin{proof}[Proof of Theorem \ref{thm:RotatingOseen_resolvent}]
The uniqueness statement is a direct consequence of Lemma \ref{lem:RotatingOseen_Uniqueness_resolvent}.
Estimate \eqref{est:RotatingOseen_resolvent} has been proved in 
Lemma \ref{lem:RotatingOseen_ResolventEst_Final}.
It thus remains to show existence of a solution for $\vf\in\LR{q}(\Omega)$.
Consider a sequence $\np{\vf_j}\subset\CRci(\Omega)$ that converges to $\vf$ in $\LR{q}(\Omega)$.
By Lemma \ref{lem:RotatingOseen_Existence}, for each $j\in\N$ 
there exists a solution $\np{\vvel,\vpres}=\np{\vvel_j,\vpres_j}$ to \eqref{sys:RotatingOseen_resolvent}
with $\vf=\vf_j$, which obeys estimate \eqref{est:RotatingOseen_resolvent} by 
Lemma \ref{lem:RotatingOseen_ResolventEst_Final}.
Additionally, this implies that $(\vvel_j,\grad\vpres_j)$ is a Cauchy sequence
in the function space defined by the norm on the left-hand side of \eqref{est:RotatingOseen_resolvent},
and thus possesses a limit $(\vvel,\grad\vpres)$,
which satisfies \eqref{sys:RotatingOseen_resolvent} and \eqref{est:RotatingOseen_resolvent}.
\end{proof}

\subsection{The time-periodic linear problem}

\begin{proof}[Proof of Theorem \ref{thm:RotatingOseen_linear}]
An application of the Fourier transform $\FT_\torus$ on $\torus$ to \eqref{sys:RotatingOseen_linear} 
reduces the uniqueness statement to the corresponding uniqueness result for the resolvent problem 
established in Theorem \ref{thm:RotatingOseen_resolvent}.
To show existence, consider $f\in\AR(\torus;\LR{q}(\Omega))$. Then
\[
f(t,x)=
\sum_{k\in\Z} f_k(x) \e^{ikt}
\] 
with $f_k\in\LR{q}(\Omega)$.
Let $\np{\uvel_k,\upres_k}=\np{\vvel,\vpres}$ be a solution to the resolvent problem 
\eqref{sys:RotatingOseen_resolvent} with $\vf=\uf_k$ 
that exists due to 
Theorem \ref{thm:RotatingOseen_resolvent}.
We define 
\[
\uvel(t,x)\coloneqq\sum_{k\in\Z} \uvel_k(x) \e^{ikt}, 
\qquad
\upres(t,x)\coloneqq\sum_{k\in\Z} \upres_k(x) \e^{ikt}.
\]
By \eqref{est:RotatingOseen_resolvent}, 
$\uvel$ and $\upres$ are well defined and satisfy \eqref{sys:RotatingOseen_linear}. 
We directly conclude estimate
\eqref{est:RotatingOseen_linear} from estimate \eqref{est:RotatingOseen_resolvent}.
\end{proof}

\section{The nonlinear problem}\label{sec:ProofNonlinear}

We return to the nonlinear problem \eqref{sys:RotatingOseen_Dimensionless}.
At first, we reformulate it as a problem with homogeneous boundary conditions.
To this end, fix $R>0$ such that $\partial\ball_R\subset\Omega$.
Let $\varphi\in\CRci(\R^3)$ be a smooth function satisfying 
$\varphi(x)=1$ if $\snorm{x}<R$, and $\varphi(x)=0$ if $\snorm{x}>2R$,
and define
\[
\Uvel\colon\torus\times\R^3\to\R^3, \qquad
\Uvel(t,x)=\frac{1}{2}\rot\bb{\bp{\veltrans(t)\eone\wedge x-\tay \eone\snorm{x}^2}\varphi(x)}.
\]  
Then $\Uvel(t,\cdot)\in\CRci(\R^3)$ for all $t\in\torus$, 
$\Uvel\in\CR{1}(\torus\times\R^3)$,
$\Div\Uvel=0$, 
and a brief calculation shows $\Uvel(t,x)=\veltrans(t)\eone+\tay\eone\wedge x$ 
for $(t,x)\in\torus\times\partial\Omega$.
Now define
$\vvel\coloneqq\uvel-\Uvel$
and $\vpres\coloneqq\upres$.
Then $\np{\uvel,\upres}$ solves \eqref{sys:RotatingOseen_Dimensionless}
if and only if $\np{\vvel,\vpres}$ solves
\begin{align}\label{sys:RotatingOseen_AfterLifting}
\begin{pdeq}
\tay\rotderterm{\vvel} -\Delta\vvel - \rey \partial_1 \vvel +\grad\vpres&= f+\caln(\vvel)
&& \tin \torus\times\Omega, \\
\Div\vvel&=0
&& \tin \torus\times\Omega, \\
\vvel&=0
&& \ton \torus\times\partial\Omega, \\
\lim_{\snorm{x}\to\infty} \vvel(t,x) &= 0
&& \tfor t\in \torus,
\end{pdeq}
\end{align} 
where 
\begin{align*}
\caln(\vvel)
&\coloneqq
\np{\projcompl\veltrans}\partial_1\vvel
-\tay\rotderterm{\Uvel}\\
&\quad+\Delta\Uvel
+\veltrans\partial_1\Uvel
-\vvel\cdot\grad\vvel
-\Uvel\cdot\grad\vvel
-\vvel\cdot\grad\Uvel
-\Uvel\cdot\grad\Uvel.
\end{align*}
Recall that $\projcompl\alpha=\alpha-\rey$.
It thus remains to show existence of a solution to the nonlinear system
\eqref{sys:RotatingOseen_AfterLifting}.

\begin{proof}[Proof of Theorem \ref{thm:RotatingOseen_nonlinear}] 
We define the function space 
\begin{align*}
\calx^q
&\coloneqq\setcl{\vvel\in\LRloc{1}(\torus\times\Omega)}{\norm{\vvel}_{\calx^q}<\infty},
\\
\norm{\vvel}_{\calx^q}
&\coloneqq\tay\norm{\rotdertermsimple\vvel}_{\AR^{q}}
+\norm{\grad^2\vvel}_{\AR^q}
+\rey\norm{\partial_1\vvel}_{\AR^q}
+\rey^{1/2}\norm{\vvel}_{\AR^{s_1}}
+\rey^{1/4}\norm{\grad\vvel}_{\AR^{s_2}},
\end{align*}
where $s_1=2q/\np{2-q}$, $s_2=4q/\np{4-q}$ and  
\[
\norm{h}_{\AR^s}\coloneqq\norm{h}_{\AR(\torus;\LR{s}(\Omega))}.
\]
At first, we derive suitable estimates of $\caln(\vvel)$. 
For example, analogously to the proof of Proposition \ref{prop:Hoelder_AR}, we have
\[
\norm{\np{\projcompl\veltrans}\partial_1\vvel}_{\AR^q}
\leq\norm{\projcompl\veltrans}_{\AR(\torus;\R)}\norm{\partial_1\vvel}_{\AR^q}
\leq\varepsilon\norm{\partial_1\vvel}_{\AR^q}
\leq \varepsilon \rey^{-1} \norm{\vvel}_{\calx^q}.
\]
Moreover, since $\frac{4q}{4-q}\leq2\leq\frac{3q}{3-q}$, 
we can employ estimates \eqref{est:Hoelder_AR} and \eqref{est:Interpolation_AR}
to obtain
\[
\norm{\vvel\cdot\grad\vvel}_{\AR^q}
\leq\norm{\vvel}_{\AR^{2q/(2-q)}}\norm{\grad\vvel}_{\AR^2}
\leq\Cc{c}\norm{\vvel}_{\AR^{2q/(2-q)}}
\norm{\grad\vvel}_{\AR^{4q/(4-q)}}^{1-\theta}\norm{\grad\vvel}_{\AR^{3q/(3-q)}}^\theta
\]
with $\theta=\frac{10q-12}{q}$. 
By the Sobolev inequality we thus deduce
\[
\norm{\vvel\cdot\grad\vvel}_{\AR^q}
\leq\Cc{c}\rey^{-1/2-\np{1-\theta}/4}\norm{\vvel}_{\calx^q}^{2-\theta}
\norm{\grad^2\vvel}_{\AR^{q}}^\theta
\leq\Cc{c}\rey^{-(3q-3)/q}\norm{\vvel}_{\calx^q}^2.
\]
The remaining terms in $\caln(\vvel)$ can be estimated in a similar fashion, which results in
\begin{align}\label{est:RotatingOseen_NonlinearTerms}
\begin{split}
\norm{\caln(\vvel)}_{\AR^q}
&\leq\Cc{c}\bp{
\varepsilon\rey^{-1}\norm{\vvel}_{\calx^q}
+\rey^{-(3q-3)/q}\norm{\vvel}_{\calx^q}^2\\
&\qquad+\np{\rey+\tay+\varepsilon}
\np{1+\rey+\tay+\varepsilon+\norm{\tddt\veltrans}_{\AR(\torus;\R)}+\norm{\vvel}_{\calx^q}}}.
\end{split}
\end{align}
Now consider the problem 
\begin{align}\label{sys:RotatingOseen_FixedPointInteration}
\begin{pdeq}
\tay\rotderterm{\wvel} -\Delta\wvel - \rey \partial_1 \wvel +\grad\wpres&= f+\caln(\vvel)
&& \tin \torus\times\Omega, \\
\Div\wvel&=0
&& \tin \torus\times\Omega, \\
\wvel&=0
&& \ton \torus\times\partial\Omega,
\end{pdeq}
\end{align}
for given $\vvel\in\calx^q$. 
Due to estimate \eqref{est:RotatingOseen_NonlinearTerms} and Theorem \ref{thm:RotatingOseen_linear}
there exists a unique velocity field $\wvel\in\calx^q$ 
and a pressure field $\wpres$ with $\grad\wpres\in\AR^q$
that satisfy \eqref{sys:RotatingOseen_FixedPointInteration} and the estimate
\begin{align*}
\norm{\wvel}_{\calx^q}
\leq\const{const:RotatingOseen_linear}
\bp{\norm{f}_{\AR^q}+\norm{\caln(\vvel)}_{\AR^q}}
&\leq\Cc[const:RotatingOseen_SelfMapping]{c}\bp{
\varepsilon+\varepsilon\rey^{-1}\norm{\vvel}_{\calx^q}
+\rey^{-(3q-3)/q}\norm{\vvel}_{\calx^q}^2\\
&\qquad+\np{\rey+\tay+\varepsilon}
\np{1+\rey+\tay+\varepsilon+\norm{\tddt\veltrans}_{\AR(\torus;\R)}+\norm{\vvel}_{\calx^q}}}.
\end{align*}
We thereby obtain a solution map $\cals\colon\calx^q\to\calx^q$, $\vvel\mapsto\wvel$ which
is a self-mapping on the ball
\[
\calx^q_\delta\coloneqq\setcl{\vvel\in\calx^q}{\norm{\vvel}_{\calx^q}\leq\delta}
\]
provided 
\[
\const{const:RotatingOseen_SelfMapping}\bp{
\varepsilon
+\varepsilon\rey^{-1}\delta
+\rey^{-(3q-3)/q}\delta^2
+\np{\rey+\tay+\varepsilon}\np{1+\rey+\tay+\varepsilon+\norm{\tddt\veltrans}_{\AR(\torus;\R)}+\delta}}
\leq\delta.
\]
Recall that $\rho\in\bp{\frac{3q-3}{q},1}$. Choosing $\delta:=\lambda^\rho$, one readily verifies that there is a constant $\kappa>0$ depending on $\const{const:RotatingOseen_SelfMapping}$ such the condition above is satisfied with $\omega\leq \kappa\rey^\rho$, $\varepsilon=\rey^2$
and $\rey_0$ sufficiently small.
In the same way, one derives the estimate
\[
\norm{\caln(\vvel_1)-\caln(\vvel_2)}_{\AR^q}
\leq\Cc{c}
\bp{\varepsilon\rey^{-1}+\rey+\tay+\varepsilon
+\rey^{-(3q-3)/q}\np{\norm{\vvel_1}_{\calx^q}+\norm{\vvel_2}_{\calx^q}}}
\norm{\vvel_1-\vvel_2}_{\calx^q},
\]
which ensures that $\cals$ is a contraction on $\calx^q_\delta$ 
with a similar choice of parameters.
Finally, the contraction mapping principle yields the existence of a fixed point $\vvel\in\calx^q$
of $\cals$, and hence of a solution $\np{\vvel,\vpres}$ to \eqref{sys:RotatingOseen_AfterLifting}. 
Consequently, $\np{\uvel,\upres}\coloneqq\np{\vvel+\Uvel,\vpres}$ is a solution 
to \eqref{sys:RotatingOseen_Dimensionless}.
\end{proof}




\end{document}